\documentclass[12pt]{article}
\usepackage{amsmath, amscd, amssymb, latexsym, epsfig, color, amsthm, mathtools, tikz-cd, array, adjustbox, bbm, tikz, enumitem, relsize,comment}
\usepackage{titling}


\DeclarePairedDelimiter\floor{\lfloor}{\rfloor}
\usepackage[normalem]{ulem}

\numberwithin{equation}{section}
\newtheorem{theorem}{Theorem}[section]
\newtheorem{proposition}[theorem]{Proposition}
\newtheorem{question}[theorem]{Question}
\newtheorem{corollary}[theorem]{Corollary}

\newtheorem{lemma}[theorem]{Lemma}

\theoremstyle{definition}

\newtheorem{definition}[theorem]{Definition}
\newtheorem{example}[theorem]{Example}

\newtheorem{reminner}[theorem]{Remark}

\makeatletter
\newenvironment{rem}
{\reminner}
{\endreminner\@endpetrue}
\makeatother

\DeclareMathOperator\lk{\mathrm{lk}}

\newcommand{\field}{\mathbbm{k}}

\newcommand{\ZZ}{{\mathbb Z}}

\newcommand{\NE}{{\mathfrak{E}}}

\allowdisplaybreaks

\tikzset{
	labl/.style={anchor=south, rotate=90, inner sep=.5mm}
}

\title{Non-Eulerian Dehn--Sommerville relations}
\author{Connor Sawaske, Lei Xue\thanks{The second author's research was partially supported by a graduate fellowship from NSF grant DMS-1664865.}\\
	\small \texttt{connor@sawaske.com, lxue@uw.edu}
}

\begin{document}
	\maketitle
	
\begin{abstract}
The classical Dehn--Sommerville relations assert that the $h$-vector of an Eulerian simplicial complex is symmetric. We establish three generalizations of the Dehn--Sommerville relations: one for the $h$-vectors of pure simplicial complexes, another one for the flag $h$-vectors of balanced simplicial complexes and graded posets, and yet another one for the toric $h$-vectors of graded posets with restricted singularities. In all of these cases, we express any failure of symmetry in terms of ``errors coming from the links." For simplicial complexes, this further extends Klee's semi-Eulerian relations.
\end{abstract}

\section{Introduction}\label{sect:intro}

\indent  In this paper we generalize Dehn--Sommerville relations in three ways: the first one relates to the $h$-vectors of all pure simplicial complexes, the second one deals with the flag $h$-vectors of balanced simplicial complexes and graded posets, and the third one concerns the toric $h$-vectors.

In 1964, Klee defined Eulerian and Semi-Eulerian simplicial complexes and proved that their $h$-vectors are almost symmetric, see \cite{MR189039}. More precisely, the $h$-vector of a $(d-1)$-dimensional Eulerian simplicial complex $\Delta$ (for example, a simplicial sphere) satisfies $h_i(\Delta) =h_{d-i}(\Delta)$ for all $i$, while the $h$-vector of a $(d-1)$-dimensional semi-Eulerian complex $\Gamma$ (such as the boundary of a simplicial manifold) satisfies $h_{d-i}(\Gamma) =h_i(\Gamma) + (-1)^{i} \binom{d}{i} \left[\tilde{\chi}(\Gamma)-(-1)^{d-1}\right]$, where $\tilde{\chi}$ is the reduced Euler characteristic of $\Gamma$. Since then these relations have played a very important role in the $f$-vector theory, e.g., in the proof of the Upper Bound Theorem, see \cite{MR166682}, \cite{MR458437}, and \cite{MR1669325}. In 2012, Novik and Swartz derived similar results for {\it psuedo-manifolds with isolated singularities} as defined in \cite{MR2943717}.

A $(d-1)$-dimensional simplicial complex $\Delta$ is balanced if it has a vertex coloring in $d$ colors such that no two vertices in the same face are colored with the same color. (It is standard to label  the colors by elements of $[d]$.) A refinement of the usual $f$- and $h$- vectors for balanced complexes are called flag $f$- and flag $h$-vectors. The flag $f$-vector of a $(d-1)$-dimensional balanced simplicial complex $\Delta$, denoted $\{f_S (\Delta)\}_{S\subseteq [d]}$, counts the number of faces of $\Delta$ according to the color sets of their vertices. The flag $h$-vector of $\Delta$, denoted  $\{h_S (\Delta) \}_{S\subseteq [d]}$, is the image of the flag $f$-vector under a certain invertible linear transformation. Bayer and Billera proved the Dehn--Sommerville relations on flag $f$-vectors of Eulerian balanced simplicial complexes, see \cite{MR774533} (also see \cite[Thm.~3.16.6]{MR2868112} for the proof of the flag $h$-vector version). The Bayer--Billera relations played an instrumental role in Fine's definition of the cd-index (see \cite{MR1073071}, \cite{MR1322068}).

Stanley \cite{MR951205} (see also \cite{MR1322068}) extended Klee's definition of Eulerian and semi-Eulerian complexes to (finite) graded partially ordered sets (posets for short). He also introduced a notion of toric $h$- and $g$-vectors of posets and proved that the toric $h$-vector of any Eulerian poset is symmetric. This result was extended by Swartz \cite{MR2505437} to semi-Eulerian posets.

Here we provide generalizations of these three types of Dehn--Sommerville relations. Our results can be summarized as follows; for all undefined terminology and notations, see Section \ref{sect:prelim}.\begin{itemize}
\item For an arbitrary $(d-1)$-dimensional pure simplicial complex $\Delta$, we express $h_{d-i}(\Delta) - h_i(\Delta)$ in terms of the Euler characteristics of links of faces, see Theorem \ref{DSThm}. The result is also generalized to simplicial posets, see Corollary \ref{cor: DS for simplicial posets}.

\item For an arbitrary $(d-1)$-dimensional balanced
 simplicial complex $\Delta$ and $S\subseteq [d]$, we express $h_S(\Delta) - h_{[d]-S} (\Delta)$ in terms of the Euler characteristics of links of faces whose color sets are contained in $S$, see Theorem \ref{thm: flag DS}.
 
\item For finite posets, we define the notion of {\bf $j$-Singular} posets (see Section \ref{section: j-Sing}) such that
      \subitem $j=-1$ recovers Eulerian posets;
      \subitem $j=0$ recovers semi-Eulerian posets;
      \subitem $j=1$ is analogous to complexes with isolated singularities (see the definition in Section \ref{sect:1-Sing posets}).
      
\item Extending the results of Stanley and Swartz, for a 1-Singular poset $P$ of rank $d+1$, we express $\hat{h}(P,x)- x^d\hat{h}(P,\frac{1}{x})$ in terms of the M\"obius functions of intervals $[s,t]$ in $P$ of length greater than or equal to $d-1$, see Theorem \ref{1IS-Thm}. Here $\hat{h}$ denotes the toric $h$-polynomial.

\item We extend this result further and obtain a similar formula for a $j$-Singular poset $P$ with $j < \floor{\frac{d}{2}}$ (see Thereoms \ref{generalized Thm} and \ref{main gen Thm}).
\end{itemize}

The structure of this paper is as follows: Section \ref{sect:prelim} introduces some basic results and definitions pertaining to simplicial complexes and posets. Section \ref{sect: Simplicial DS} is devoted to establishing the generalization of Dehn--Sommerville relations for pure simplicial complexes and simplicial posets. Section \ref{sect:flag DS rel} proves the flag Dehn--Sommerville relations for balanced simplicial complexes and graded posets. Sections \ref{sect:1-Sing posets} and \ref{section: j-Sing} establish the toric generalizations of Dehn--Sommerville formulas. Our proofs build on methods used by Klee, Stanley, and Swartz.\\


\section{Preliminaries}\label{sect:prelim}

\subsection{Simplicial complexes}\label{subsect: s.c.}
In this section we review some definitions pertaining to simplicial complexes. Let $V$ be a finite set. A \textbf{simplicial complex} $\Delta$ with vertex set $V$ is a collection of subsets of $V$ that is closed under inclusion. We call each element of $\Delta$ a \textbf{face} of $\Delta$, and each face $F\in\Delta$ has a \textbf{dimension} defined by $\dim (F) =|F|-1$. Similarly, the dimension of $\Delta$ is defined by $\dim(\Delta)=\max\{\dim F: F\in\Delta\}$. If all maximal faces of $\Delta$ (with respect to inclusion) have the same dimension, then $\Delta$ is called \textbf{pure}. We denote the collection of faces of $\Delta$ of a specific dimension $i$ by
	\[
		\Delta_i:=\{F\in\Delta: \dim (F) = i\}.
	\]
Lastly, the \textbf{link} of a face $F$ of $\Delta$, denoted $\lk_\Delta F$, is defined by
	\[
	\lk_\Delta F := \{G\in \Delta: F\cup G\in \Delta \text{ and } F\cap G = \emptyset\}.
	\]

Let $\Delta$ be a simplicial complex of dimension $d-1$. The {\bf $f$-vector} of $\Delta$ is defined by $f(\Delta) := (f_{-1}(\Delta), f_0(\Delta), f_1(\Delta), \dots  , f_{d-1}(\Delta))$, where $
f_i(\Delta) = |\Delta_i|$. We further define the {\bf $h$-vector} of $\Delta$ by $h(\Delta):=(h_0(\Delta), h_1(\Delta), \dots , h_d(\Delta))$, with entries determined by the equation
\[\sum_{i=0}^d h_i(\Delta) x^{d-i} =  \sum_{i=0}^d f_{i-1}(\Delta) (x - 1)^{d-i}.\]
For the remainder of this section, we will assume that $\Delta$ is a pure simplicial complex of dimension $d-1$. 

Each simplicial complex $\Delta$ admits a geometric realization $\|\Delta\|$ that contains a geometric $i$-simplex for each $i$-face of $\Delta$. We say that $\Delta$ is a simplicial sphere (manifold, respectively) if $\|\Delta\|$ is homeomorphic to a sphere (manifold, respectively). 

The (reduced) \textbf{Euler characteristic} of $\Delta$ is
	\[
	\tilde{\chi}(\Delta) := \sum_{i=-1}^{d-1}(-1)^if_i(\Delta),
	\]
and by the Euler-Poincar\'e formula, $\tilde{\chi}(\Delta)$ is a topological invariant of $\Delta$, or more precisely, of its geometric realization $\|\Delta\|$. For instance, if $\Gamma$ is an $(i-1)$-dimensional simplicial sphere, then $\tilde{\chi}(\Gamma) = (-1)^{i-1}$.

Given two simplicial complexes $\Delta_1$ and $\Delta_2$ on disjoint vertex sets, their {\bf simplicial join}, $\Delta_1 \ast \Delta_2$, is defined as
\[\Delta_1 \ast \Delta_2 :=  \{F\cup G:\; F\in \Delta_1,\;G\in\Delta_2 \}.  \]
In particular, $\Delta_1 \ast \Delta_2$ is a simplicial complex of dimension $\dim \Delta_1 +\dim \Delta_2+1$.

Central to many classifications of simplicial complexes is the notion of the link of a face having the same Euler characteristic as that of a sphere of the appropriate dimension. To that end, we measure potential failures of this condition by defining an error function $\varepsilon_{\Delta}(F)$ on faces $F$ of a pure $(d-1)$-dimensional simplicial complex $\Delta$ as	
	\[
	\varepsilon_{\Delta}(F):= \tilde{\chi}\left(\lk_\Delta F\right)- (-1)^{d-1-|F|}
	\]
(note that $\dim(\lk_{\Delta} F) = d-1-|F|$, so $(-1)^{d-1-|F|}$ is the same as the reduced Euler characteristic of a sphere of dimension $\dim(\lk_{\Delta} F )$). In addition, we form the set of faces with non-trivial error as
	\[
		\NE(\Delta) :=\{F\in\Delta: \varepsilon(F)\not=0\}.
	\]
We say that $\Delta$ is \textbf{Eulerian} if $\NE(\Delta)=\emptyset$, and that it is \textbf{semi-Eulerian} if $\NE(\Delta)=\{\emptyset\}$. In line with these definitions, we refer to $\NE(\Delta)$ as the \textbf{non-Eulerian} part of $\Delta$.

\begin{example}
	If $\Delta$ is a simplicial sphere, then $\NE(\Delta)=\emptyset$. More generally, if $\Delta$ is a simplicial manifold, then $\NE(\Delta)\subseteq\{\emptyset\}$.
\end{example}

\subsection{Graded posets}\label{subsect: poset}
The notions of being Eulerian and semi-Eulerian can also be defined in the graded poset settings. Throughout this paper, we let every finite graded poset $P$ have unique bottom and top elements $\hat{0}$ and $\hat{1}$. Let $\rho:P\to \mathbb{N}$ be the rank function. The rank of $P$, $\rho(P)$, is defined as $\rho(\hat{1})$.

Let $\mu_P$ denote the M\"obius function of poset $P$. If for all proper intervals $[s,t] \subsetneq P$, $\mu_P(s,t) = (-1)^{\rho(t)-\rho(s)}$, then $P$ is called {\bf semi-Eulerian}. If in addition, $\mu_P(\hat{0},\hat{1})= (-1)^{\rho(P)}$, then $P$ is {\bf Eulerian}.

For each poset $P$, there is a simplicial complex associated with P; it is called the (reduced) {\bf order complex} of $P$ and denoted by $O(P)$, see for instance \cite{MR1373690}. The complex $O(P)$ has the set $P\setminus \{\hat{0}, \hat{1} \}$ as its vertex set, and the (finite) chains in the open interval $(\hat{0}, \hat{1} )$ as its faces. A chain $C=\{t_1<t_2<\cdots<t_k\}$ in $P\setminus \{\hat{0},\hat{1}\}$ of size $k$ corresponds to a face in $O(P)$ of dimension $k-1$; this face is sometimes denoted by $F_C$. In particular, if $\rho(P)=d+1$, then $\dim O(P)=d-1$.

	\begin{rem}[Relations between $\mu$ and $\tilde{\chi}$] For a poset $P$ and a chain $C$ in it, the following formulas hold:
		\begin{align}\label{eqn: chi and mu 1}
		\tilde{\chi}(O(P)) = \mu_P(\hat{0},\hat{1});
		\end{align}
		\begin{align}\label{eqn: chi and mu 2}
		\tilde{\chi}(\lk_{O(P)} F_C) = (-1)^{|F|} \mu_P(\hat{0},t_1)\cdot \mu_P(t_1,t_2)\dots \cdot \mu_P(t_{k-1},t_k)\mu_P(t_k,\hat{1}).
		\end{align}
	\end{rem}
	The first formula is well known, see, for instance, \cite{MR2868112}. To prove the second equality, note that
	\[
	\lk_{O(P)} F = O(\hat{0},t_1) \ast O(t_1,t_2) \dots \ast  O(t_{k-1},t_k)\ast O(t_k,\hat{1}), 
	\]
	where $O(s,t)$ is the order complex of the open interval $(s,t) = \{x:\; s<x<t \}$. Using this and the fact that $\tilde{\chi}(\Delta_1 \ast \Delta_2) = (-1)\tilde{\chi}(\Delta_1) \tilde{\chi}(\Delta_2)$, we obtain
	\begin{align*}
	\tilde{\chi}(\lk_{O(P)} F) = (-1)^{|F|} \tilde{\chi}(O(\hat{0},t_1))\cdot \tilde{\chi}(O(t_1,t_2))\dots \cdot \tilde{\chi}(O(t_{k-1},t_k))\tilde{\chi}(O(t_k,\hat{1})).
	\end{align*}
	This together with (\ref{eqn: chi and mu 1}) implies (\ref{eqn: chi and mu 2}). 
	
	Since each face $F_C\in O(P)$ corresponds to a chain $C\in P \backslash \{\hat{0}, \hat{1} \}$, we can define a {\bf chain error} in $P$ that corresponds to the face error in $O(P)$:
	Let $C = \{t_1<\dots <t_{i}\}$ be a chain in $P\backslash \{\hat{0}, \hat{1} \}$. Define
	\[
	\mu_P(C):= \mu_P(\hat{0},t_1) \mu_P(t_1,t_2)\dots \mu_P(t_{i},\hat{1}), 
	\]
	and
	\[ 
	\varepsilon_P(C):= (-1)^{|C|} [\mu_P(C) - (-1)^{d+1}].
	\]
	We call $C \mapsto \varepsilon_P(C)$ the {\bf error function for chains} in a poset $P$. However this is not a ``new'' error function: comparing it with $\varepsilon_{O(P)}(-)$ and using (\ref{eqn: chi and mu 2}), we obtain:
	
	\begin{rem}\label{rmk: chain error = link error}
		$\varepsilon_P(C) = \varepsilon_{O(P)}(F_C)$.
	\end{rem}

\subsection{Balanced simplicial complexes and flag vectors}
For a specific type of simplicial complexes, called the balanced simplicial complexes, there exists a certain refinement of $f$- and $h$-vectors. A $(d-1)$-dimensional pure simplicial complex $\Delta$ is {\bf balanced} if it is equipped with a vertex coloring $\kappa: V \to [d]$ such that no two vertices in the same face have the same color. These complexes were introduced by Stanley in \cite{MR526314}. 

For any subset $S \subseteq [d]$, the {\bf $S$-rank selected subcomplex of $\Delta$} is 
\[\Delta_S =\{F \in \Delta:\;\kappa(F)\subseteq S\}. \]

Define $f_S(\Delta)$ as the number of faces in $\Delta$ with $\kappa(F) = S$. The numbers $f_S(\Delta)$ are called the flag $f$-numbers of $\Delta$ and the collection $(f_S(\Delta))_{S\subseteq [d]}$ is called the {\bf flag $f$-vector of $\Delta$}.
Similarly, the flag $h$-numbers of $\Delta$ are defined as
\[ 
h_T (\Delta) = \sum_{S\subseteq T} (-1)^{|T|-|S|} f_S(\Delta) \quad \text{for } T \subseteq [d]. 
\]
Note that the flag $f$- and $h$-numbers refine the ordinary $f$- and $h$-numbers:
\[ 
f_{i-1}(\Delta) = \sum_{S\subseteq [d],\; |S|=i} f_S(\Delta)\quad \text{ and } \quad h_j (\Delta) =  \sum_{T\subseteq [d],\; |T|=j} h_T (\Delta). 
\]

One common example of a balanced complex is the order complex of a graded poset: let P be a graded poset of rank $d+1$, then $O(P)$ is balanced w.r.t.~the coloring given by the rank function of $P$. Moreover, the flag vectors can also be defined in the setting of graded posets. For $S\subseteq [d]$, we define
\[
P_{S}=\left\{x\in P:\rho(x)\in S\cup \{0, d+1\} \right\} 
\] 
considered as a subposet of $P$. The poset $P_S$ is called the {\bf S-rank selected subposet} of $P$ and if $\Delta = O(P)$, then $\Delta_S = O(P_S)$. We let $ \alpha _{P}(S)$ be the number of maximal (w.r.t.~inclusion) chains in $P_S$. The function $S\mapsto \alpha _{P}(S)$
is the {\bf flag $f$-vector of P}. We also consider the function
\[
S\mapsto \beta _{P}(S),\quad \beta _{P}(S)=\sum _{T\subseteq S}(-1)^{|S|-|T|}\alpha _{P}(T)
\]
and call it the {\bf flag $h$-vector of P}. It is easy to see that 
\begin{align}
\alpha_P(S) = f_S(O(P)) \quad \text{and } \quad \beta_P(S) =h_S(O(P)).
\end{align}

\subsection{The Stanley-Reisner ring}
An equivalent way to define the $h$- and flag $h$-vectors is through the Stanley-Reisner ring (for more details see \cite[Ch.~II.1]{St-96}). Let $\Delta$ be a $(d-1)$-dimensional simplicial complex with vertex set $V = [n]$. Let $\field$ be a field and let $R = \field[x_1, \dots, x_n]$. The {\bf Stanley-Reisner ring of $\Delta$} is $\field[\Delta] = R/I_{\Delta}$, where
\[I_{\Delta} = \langle x_{i_1} x_{i_2} \dots \cdot  x_{i_k} : \; \{i_1, \dots, i_k \} \notin \Delta  \rangle. \]

Let $\field [\Delta]_i$ be the $i$-th homogeneous component of $\field[\Delta]$. {\bf The Hilbert series of $\field[\Delta]$} is
\begin{align*}
	F(\Delta, \lambda) = \sum_{i=0}^{\infty} \dim_{\field} \field[\Delta]_i \; \lambda^i.
\end{align*}
The $h$-vector of $\Delta$ can be easily obtained from the Hilbert series of $\field[\Delta]$ using the following relation (see, for example, \cite[Ch.~II.2]{St-96}):
\begin{align} 
	F(\Delta, \lambda) =\frac{\sum_{i=0}^{d} h_i(\Delta)\lambda^i} {(1 - \lambda)^d}.
\end{align}

If $\Delta$ is balanced with vertex coloring $\kappa:[n] \to [d]$, then $\field [\Delta]$  has a natural $\ZZ^d
$-grading which is induced by this coloring. For $i = 1,2,\dots, d$, let $e_i \in \{0,1 \}^d$ be the $i$-th coordinate unit vector, and for $j\in [n]$, define $\deg(x_j) = e_{\kappa(j)}$. This gives a $\ZZ^d$ grading of $\field[x_1,\dots, x_n]$. Let $\lambda = (\lambda_1, \dots, \lambda_d)$ be a $d$-tuple of variables. For $a\in\ZZ^d_{\geq 0}$, we let $\lambda^a = \lambda_1^{a_1}\lambda_2^{a_2}\dots \lambda_d^{a_d}$. The {\bf fine Hilbert series} of $\field[\Delta]$ with respect to this $\ZZ^d$ grading is
\[
	F(\Delta, \lambda) = \sum_{a\in \ZZ^d_{\geq 0}} (\dim_{\field} \field[\Delta]_a)\cdot \lambda^a. 
\]
The flag $h$-vectors can be obtained from this fine Hilbert series (see \cite{MR526314}):
\begin{align}  F(\Delta, \lambda) = \frac{1} {\prod_{i=1}^{d}(1 - \lambda_i)} \sum_{S\subseteq [d]} h_S \lambda^S
\end{align}
where $\lambda^S = \prod_{j\in S} \lambda_j$.


\subsection{Toric vectors of graded posets}
We will encounter toric vectors only in Sections \ref{sect:1-Sing posets} and \ref{section: j-Sing}, so the reader may skip this subsection for now and come to it later.

As in Subsection \ref{subsect: poset}, we let $P$ be a finite graded poset with unique bottom and top elements $\hat{0}$ and $\hat{1}$. If $P$ has only one element, i.e., when $\hat{0} =\hat{1}$, then we call $P$ the trivial poset and denote it as $P=\mathbbm{1}$. Let $\rho:P\to \mathbb{N}$ be the rank function. Let $\widetilde{P} = \{[\hat{0},t]:\; t\in P \}$ be the poset of lower intervals of $P$ ordered by inclusion. Define two polynomials $\hat{h}(P,x)$ and $\hat{g}(P,x)$ recursively as follows.
\begin{itemize}
	\item $\hat{h}(\mathbbm{1},x)= \hat{g}(\mathbbm{1},x) =1$.
	\item If $P$ has rank $d+1$, then $\deg \hat{h}(P,x)=d$. We first write \[\hat{h}(P,x) = \hat{h}_d + \hat{h}_{d-1} x + \hat{h}_{d-2} x^2 +\dots +\hat{h}_{0} x^d.\]
	We then define $\hat{g}(P,x)$ as
	\[\hat{g}(P,x) := \hat{h}_d + (\hat{h}_{d-1} -\hat{h}_d)x + (\hat{h}_{d-2} -\hat{h}_{d-1}) x^2 +\dots +(\hat{h}_{d-m}-\hat{h}_{d-m+1})  x^m,  \]
	where $m =\deg \hat{g}(P,x)=\floor{\frac{d}{2}}$.
	\item Finally, for a poset $P$ of rank $d+1$, define
	\[\hat{h}(P,x) :=\sum_{Q\in \widetilde{P}, Q\neq P} \hat{g}(Q,x)(x-1)^{d-\rho(Q)}.  \]
\end{itemize}
The coefficients of these polynomials, arranged as vectors, are called the {\bf toric $h$-vector} and the {\bf toric $g$-vector}, respectively.

\begin{rem}
	We follow the convention of Swartz \cite{MR2505437}, and so our $\hat{h}_i$ is $\hat{h}_{d-i}$ in Stanley's paper \cite{MR951205}.
\end{rem}


\section{Dehn--Sommerville relations}\label{sect: Simplicial DS}
\subsection{Pure Simplicial Complexes}
The main result of this section is the following generalization of Dehn--Sommerville relations (see \cite{MR166682}) to all pure simplicial complexes. We will discuss two proofs of this result; the third one is sketched in Section \ref{sect:flag DS rel} (see Remark \ref{remark: third pf}).

\begin{theorem}\label{DSThm}
	Let $\Delta$ be a pure $(d-1)$-dimensional simplicial complex. Then
		\begin{align}\label{eqn: simplicial DS formula}
			h_{d-j}(\Delta)-h_j(\Delta)=(-1)^j\sum_{F\in \Delta}\binom{d-|F|}{j}\varepsilon_{\Delta}(F)\qquad \text{for }j=0, \ldots, d.
      \end{align}
\end{theorem}

\begin{proof}
	Note that $(-1)^{d-1-i}f_{i-1}(\Delta)-\sum_{F\in\Delta_{i-1}} (-1)^{d-1-|F|} =0$, and hence
		\begin{equation}\label{links1}
			\sum_{F\in\Delta_{i-1}} \tilde{\chi}(\lk_\Delta F)=(-1)^{d-1-i}f_{i-1}(\Delta)+	\sum_{F\in\Delta_{i-1}}\left[ \tilde{\chi}(\lk_\Delta F)-(-1)^{d-1-|F|}\right].
		\end{equation}
	On the other hand, since each $(j-1)$-dimensional face of $\Delta$ contains exactly $\binom{j}{i}$ faces of dimension $i-1$,
		\begin{equation}\label{links2}
			\sum_{F\in\Delta_{i-1}} \tilde{\chi}(\lk_\Delta F)=\sum_{j=i}^d(-1)^{j-i-1}\binom{j}{i}f_{j-1}(\Delta).
		\end{equation}
	
	Setting the right-hand sides of \eqref{links1} and \eqref{links2} equal to each other and multiplying throughout by $(-1)^{d-1-i}$ yields
		\begin{equation}\label{baseeqn1}
			\sum_{j=i}^d(-1)^{d-j}\binom{j}{i}f_{j-1}(\Delta)=f_{i-1}(\Delta)+(-1)^{d-1-i}\sum_{F\in\Delta_{i-1}}\varepsilon_{\Delta}(F).
		\end{equation}

	Now we multiply both sides of \eqref{baseeqn1} by $(\lambda-1)^{d-i}$ and sum the result over $i$:
		\begin{equation}\label{baseeqn2}
			\sum_{i=0}^d\left[\sum_{j=i}^d(-1)^{d-j}\binom{j}{i}f_{j-1}(\Delta)\right](\lambda-1)^{d-i}=\sum_{i=0}^d\left[f_{i-1}(\Delta)+(-1)^{d-1-i}\sum_{\mathclap{F\in \Delta_{i-1}}}\varepsilon_{\Delta}(F)\right](\lambda-1)^{d-i}.
		\end{equation}

	The left hand-side of equation \eqref{baseeqn2} may be rewritten as
		\begin{align*}
			\sum_{j=0}^d(-1)^{d-j}(\lambda-1)^{d-j}f_{j-1}(\Delta)\left(\sum_{i=0}^d\binom{j}{i}(\lambda-1)^{j-i}\right)
			&=\sum_{j=0}^d f_{j-1}(\Delta)(1-\lambda)^{d-j}\lambda^j\\
			&=\sum_{j=0}^d h_j(\Delta)\lambda^j.
		\end{align*}
	
	The right hand-side of equation \eqref{baseeqn2} can be broken up as
		\begin{align*}
			\sum_{i=0}^d\left[f_{i-1}(\Delta)+(-1)^{d-1-i}\sum_{\mathclap{F\in \Delta_{i-1}}}\varepsilon_{\Delta}(F)\right](\lambda-1)^{d-i}&=\sum_{i=0}^df_{i-1}(\Delta)(\lambda-1)^{d-i}\\
			&+\sum_{i=0}^d(-1)^{d-1-i}\left(\sum_{F\in \Delta_{i-1}}\varepsilon_{\Delta}(F)\right)(\lambda-1)^{d-i}.
		\end{align*}
	We will analyze each of these terms on the right independently. Firstly,
		\[
			\sum_{i=0}^d f_{i-1}(\Delta)(\lambda-1)^{d-i}=\sum_{i=0}^d h_i(\Delta)\lambda^{d-i}.
		\]
	For the second term,
		\begin{align*}
			\sum_{i=0}^d (-1)^{d-1-i} & \left(\sum_{F\in \Delta_{i-1}}  \varepsilon_{\Delta}(F)\right) (\lambda-1)^{d-i}\\
			&=\sum_{i=0}^d(-1)^{d-1-i}\left(\sum_{F\in \Delta_{i-1}}\varepsilon_{\Delta}(F)\right)\left(\sum_{j=0}^{d-i}(-1)^{d-i-j}\binom{d-i}{j}\lambda^j\right)\\
			&=\sum_{i=0}^d\left(\sum_{F\in \Delta_{i-1}}\varepsilon_{\Delta}(F)\right)\left(\sum_{j=0}^{d-i}(-1)^{j-1}\binom{d-i}{j}\lambda^j\right)\\
			&=\sum_{j=0}^d(-1)^{j-1}\left(\sum_{i=0}^d\binom{d-i}{j}\left(\sum_{F\in \Delta_{i-1}}\varepsilon_{\Delta}(F)\right)\right)\lambda^j.
		\end{align*}
	
	By equating coefficients in equation \eqref{baseeqn2} we obtain
		\[
			h_{d-j}(\Delta)-h_j(\Delta)=(-1)^j\left(\sum_{i=0}^d\binom{d-i}{j}\sum_{F\in \Delta_{i-1}}\varepsilon_{\Delta}(F)\right),
		\]
	and the summations on the right may be re-written as in the statement of the theorem.
\end{proof}

Since $\varepsilon_{\Delta}(F)=0$ unless $F\in\NE(\Delta)$, we have the following corollary that phrases the relationship between $h_j(\Delta)$ and $h_{d-j}(\Delta)$ in terms of the non-Eulerian part of $\Delta.$

\begin{corollary}
	Let $\Delta$ be a pure $(d-1)$-dimensional simplicial complex. Then	
		\[
			h_{d-j}(\Delta)-h_j(\Delta)=(-1)^j\sum_{F\in \NE(\Delta)}\binom{d-|F|}{j}\varepsilon_{\Delta}(F).
		\]
\end{corollary}

\begin{example}
	When $\Delta$ is semi-Eulerian (so that $\NE(\Delta)=\{\emptyset\}$),
		\[
			\sum_{F\in\NE(\Delta)}\binom{d-|F|}{j}\varepsilon_{\Delta}(F)=\binom{d}{j}\left[\tilde{\chi}(\Delta)-(-1)^{d-1}\right]
		\]
	and
		\[
			h_{d-j}(\Delta)-h_j(\Delta)=(-1)^j\binom{d}{j}\left[\tilde{\chi}(\Delta)-(-1)^{d-1}\right].
		\]
	Thus, in this case Theorem \ref{DSThm} reduces to Klee's Dehn--Sommerville equations in \cite{MR166682}.
\end{example}

\begin{example}
	As for $h$-vectors of complexes with $\NE(\Delta)$ containing faces of dimension larger than $-1$, consider the case in which $\lVert \Delta \rVert=\mathbb{S}^1*\lVert M\rVert$, where  $\mathbb{S}^1$ denotes the $1$-dimensional sphere and $M$ is some $(d-3)$-dimensional simplicial manifold with $\varepsilon_{\Delta}(\emptyset) = \tilde{\chi}(\Delta) - (-1)^{d-1}\not=0$. Then $\NE(\Delta)$ forms a cycle (in the graph theory sense), say of length $n$, and
	\footnotesize
		\begin{align*}
			\sum_{i=0}^d & \binom{d-i}{j}\left(\sum_{F\in \NE(\Delta)_{i-1}}\varepsilon_{\Delta}(F)\right)\\
			&=\binom{d}{j}\left[\tilde{\chi}(M)-(-1)^{d-1}\right]+n\binom{d-1}{j}\left[-\tilde{\chi}(M)-(-1)^{d-2}\right]+n\binom{d-2}{j}\left[\tilde{\chi}(M)-(-1)^{d-3}\right]\\
			&=\left(\tilde{\chi}(M)+(-1)^d\right)\left[\binom{d}{j}-n\binom{d-2}{j-1}\right],
		\end{align*}
	\normalsize
	and so
		\begin{align}\label{eq: example }
			h_{d-j}(\Delta)-h_{j}(\Delta)=(-1)^j\left(\tilde{\chi}(M)+(-1)^d\right)\left[\binom{d}{j}-n\binom{d-2}{j-1}\right]
		\end{align}
	for $j=0, \ldots, d$. In particular, if $M$ is a triangulation of the torus, then $d=5$ and $\tilde{\chi}(M)=-1$, and so
		\[
			h_{5-j}(\Delta)-h_j(\Delta)=(-1)^j(-2)\left[\binom{5}{j}-n\binom{3}{j-1}\right].
		\]
\end{example}

\begin{rem}
	One consequence of (\ref{eq: example }) is that, with $\tilde{\chi}(M)$ known, the exact number of non-Eulerian edges in any triangulation $\Delta$ of $\mathbb{S}^1*M$ is determined by just the face numbers  $f_i(\Delta)$ up to a dimension about $\frac{d}{2}$. 
\end{rem}

The following is an alternative proof of Theorem \ref{DSThm}. This proof uses {\it short $h$-numbers}, and we define them as follows.

For $0\leq i\leq d-1$, the {\bf $i$-th short $h$-number} (defined by Hersh and Novik) is 
	\[
		h^*_i(\Delta) = \sum_{v\in V(\Delta)} h_i(\lk_\Delta v).  
	\]
These numbers go back to McMullen's proof of the Upper Bound Theorem (see \cite{MR283691}), but were  formalized by Hersh and Novik in \cite{MR1923232}. The following formula that connects the short $h$-numbers to the usual ones was verified by McMullen for simplicial polytopes (see \cite[pg.~183]{MR283691}) and by Swartz for pure simplicial complexes (see \cite[Prop. 2.3]{MR2134424}):
	 \begin{align}\label{eq: short h and h}
		h^*_{i-1} = i h_i + (d-i+1) h_{i-1} \quad \text{ for all } 1\leq i\leq  d.
	\end{align}
	
\begin{proof}[Another Proof of Theorem \ref{DSThm}]
	We will prove (\ref{eqn: simplicial DS formula}) by double induction: first we induct on the dimension of $\Delta$, and for each $(d-1)$-dimensional simplicial complex we induct on  $i\geq 0$ (using the validity of the statement for $h_{d-i+1}-h_{i-1}$ to derive its validity for $h_{d-i}-h_i$). 
	
	If $\dim(\Delta)=0$, it is easy to check that (\ref{eqn: simplicial DS formula}) holds.
	
	If $\dim(\Delta)=d-1$ for $d>1$ and $\Delta$ is pure, then $\lk_\Delta v$ is a pure $(d-2)$-dimensional simplicial complex (for any vertex $v$). By the inductive hypothesis, 
	\begin{align}
		h_{d-i-1}(\lk_\Delta v) - h_{i}(\lk_\Delta v)  = (-1)^i \sum_{F\in \lk(v)}{d-1-|F| \choose i }\epsilon_{\lk(v)}(F),
	\end{align}
	and by summing over $v\in V(\Delta)$ on both sides, we obtain that \begin{align}\label{eq: short h DS}
		h^*_{d-i-1}(\Delta) - h^*_{i}(\Delta)  = (-1)^i \sum_{v\in V(\Delta)} \sum_{F\in \lk(v)}{d-1-|F| \choose i }\varepsilon_{\lk(v)}(F).
	\end{align}
	
Assume that (\ref{eqn: simplicial DS formula}) of Theorem \ref{DSThm} holds for all $d'<d$. The base case of the induction on $i$ is when $i=0$. By the definition of $h$-vectors and Euler's formula,
	\begin{align}\label{eqn: LHS of mormula}
		h_d(\Delta) -h_0(\Delta) =(-1)^{d-1}\tilde{\chi}(\Delta) -1.
	\end{align}
On the other hand, the expression on the right-hand side of (\ref{eqn: simplicial DS formula}) can be rewritten as
	\begin{align*}
		&(-1)^0 \sum_{F\in \Delta}{d-1-|F| \choose 0}\left[ \tilde{\chi}(\lk_{\Delta} F) -(-1)^{d-1-|F|} \right] \\ =&\sum_{i=0}^{d}\sum_{F\in \Delta_{i-1}}\tilde{\chi}(\lk_{\Delta} F) -\sum_{i=0}^{d}\sum_{F\in \Delta_{i-1}}(-1)^{d-1-i} \\
		=&\sum_{i=0}^{d}\sum_{j=i}^{d}(-1)^{j-i-1}{j \choose i} f_{j-1}(\Delta) -\sum_{i=0}^{d}(-1)^{d-1-i} f_{i-1}(\Delta)\\
		=&-\sum_{j=0}^{d} f_{j-1}(\Delta) \left[\sum_{i=0}^{j}(-1)^{j-i} {j \choose i} \right]-\sum_{i=0}^{d}(-1)^{d-1-i} f_{i-1}(\Delta)\\
		\overset{(\bowtie)}{=}&-1+(-1)^{d-1} \sum_{i=0}^{d}(-1)^{i-1} f_{i-1}(\Delta)\\
		=&-1+(-1)^{d-1}\tilde{\chi}(\Delta),\\
		\overset{(\ref{eqn: LHS of mormula})}{=}& h_d(\Delta)-h_0(\Delta)
	\end{align*}
where the equality ($\bowtie$) follows from the following (well-known) binomial identity.
	\[
		\sum_{i=0}^{j}(-1)^{j-i} {j \choose i}  = 
			\begin{cases} 
					1,& \text{if } j =0,\\
					0,& \text{if } j >0.
			\end{cases}
	\]
This completes the proof of the $i=0$ case.
	
Let $i>0$ and assume $h_{d-i+1 }(\Delta) - h_{i-1}(\Delta)$ satisfies the formula in the statement of the theorem. Then by (\ref{eq: short h DS}) and (\ref{eq: short h and h}), we have
	\footnotesize
	\begin{align*}
		i(h_{d-i} - h_i) &\stackrel{(\ref{eq: short h and h})}{=} (-1)(d-i+1)[h_{d-i+1} - h_{i-1}] + (h^*_{d-i} - h^*_{i-1})\\
		&\stackrel{\text{ind. hyp.}}{=} (-1)(d-i+1)\left[(-1)^{i-1} \sum_{F\in \Delta} {d-|F| \choose i-1  } \varepsilon_{\Delta}(F) \right] + (h^*_{d-i} - h^*_{i-1})\\
		&\stackrel{(\ref{eq: short h DS})}{=}(-1)(d-i+1)\left[(-1)^{i-1} \sum_{F\in \Delta} {d-|F| \choose i-1  } \varepsilon_{\Delta}(F) \right] \\
		&\quad\quad + (-1)^{i-1}\bigg(\sum_{v\in V(\Delta)} \sum_{F\in \lk_\Delta v} { d-1 - |F| \choose i-1} \varepsilon_{\lk (v)} (F) \bigg).
	\end{align*}
	\normalsize
Therefore
	\begin{align}\label{eq: messy formula}
		h_{d-i}(\Delta) - h_i(\Delta) &= (-1)^i\frac{d-i+1}{i}\sum_{F\in \Delta} {d-|F| \choose i-1  } \varepsilon_{\Delta}(F) \nonumber\\
		& \quad+ \frac{(-1)^{i-1}}{i}\sum_{v\in V(\Delta)}\sum_{F\in \lk_\Delta v} { d-1 - |F| \choose i-1} \varepsilon_{\lk (v)} (F).
	\end{align}
To show (\ref{eq: messy formula}) equals the right-hand side of (\ref{eqn: simplicial DS formula}), we need to show:
	\footnotesize 
	\begin{align}\label{eq: need to show}
		\frac{d-i+1}{i}\sum_{F\in \Delta} {d-|F| \choose i-1  } \varepsilon_{\Delta}(F)  - \sum_{v\in V(\Delta)}\sum_{F\in \lk_\Delta v} \frac{1}{i}{ d-1 - |F| \choose i-1} \varepsilon_{\lk (v)} (F) = \sum_{F\in\Delta} {d-|F| \choose i} \varepsilon_\Delta(F). 
	\end{align}
\normalsize
Notice that for each $v\in V(\Delta)$ and $F\in \lk_{\Delta}(v)$, $\lk_{\lk(v)} (F)  =  \lk_{\Delta} (F\cup v)$, therefore $\varepsilon_{\lk(v)} F = \varepsilon_\Delta (F \cup v)$, and so
	\[
	\sum_{v\in V(\Delta)}\sum_{F\in \lk_\Delta v}\frac{1}{i} { d-1 - |F| \choose i-1} \varepsilon_{\lk (v)} (F) = \sum_{G\in \Delta}|G| \cdot \frac{1}{i} { d- |G| \choose i-1} \varepsilon_{\Delta} (G). 
	\]
Pluging this expression into the left-hand side of (\ref{eq: need to show}), we obtain that the left-hand side of (\ref{eq: need to show}) can be rewritten as
	\begin{align*}
		 &\frac{d-i+1}{i}\sum_{F\in \Delta} {d-|F| \choose i-1  } \varepsilon_{\Delta}(F)  - \sum_{F\in \Delta}|F| \cdot \frac{1}{i} { d- |F| \choose i-1} \varepsilon_{\Delta} (F) \\
		 = &\sum_{F\in \Delta} \left[\frac{d-i+1}{i}  -   \frac{|F|}{i} \right]{ d- |F| \choose i-1} \varepsilon_{\Delta} (F)\\
		= & \sum_{F\in \Delta} {d-|F| \choose i} \varepsilon_{\Delta} (F),
	\end{align*}
and so (\ref{eq: need to show}) does hold. This completes the proof of the theorem.\end{proof}

\subsection{Simplicial posets}
	In this subsection we will show that Theorem \ref{DSThm} can be generalized to {\bf simplicial posets} (for more details see \cite{MR1117642}). A graded poset $P$ (with the unique bottom and top elements $\hat{0}$ and $\hat{1}$) is {\bf simplicial} if every proper lower interval $[\hat{0}, t]$ is a Boolean lattice. Given a simplicial complex $\Delta$, the poset of faces of $\Delta$ ordered by inclusion is a simplicial poset. (It is called the {\bf face lattice of $\Delta$}.) Therefore simplicial posets are generalizations of simplicial complexes. Many notions and structures on simplicial complexes can be generalized to simplicial posets.
	
	Given a simplicial poset $P$ of rank $d+1$, for $-1 \leq i\leq d-1$, define $f_i= f_i(P)$ as the number of elements in $P$ with rank $i+1$. The vector $f(P) = (f_{-1}, f_0, f_1,\dots , f_{d-1})$ is called the {\bf $f$-vector of $P$}. Similar to the definition of $h$-vectors in Subsection \ref{subsect: s.c.},  we can define $h_0,h_1,\dots ,h_d$ by
		\[ 
			\sum_{i=0}^d h_{i} x^{d-i} =\sum_{i=0}^d f_{i-1} (x-1)^{d-i}. 
		 \]
	The vector $h(P) = (h_0,h_1,\dots , h_d)$ is the {\bf $h$-vector of $P$}. When $P$ is the face lattice of a simplicial complex $\Delta$, then $f(P) = f(\Delta)$ and $h(P) = h(\Delta)$.
	
	The notion of links can also be generalized from simplicial complexes to simplicial posets. Let $P$ be a simplicial poset and $t\in P$, the {\bf link of $t$ in $P$} is simply the upper interval $[t,\hat{1}]$. It is easy to see that $[t,\hat{1}]$ is also a simplicial poset.
	
	With these notions in hand, both proofs of Theorem \ref{DSThm} can be easily adapted to the more general setting of simplicial posets and result in the following corollary:
\begin{corollary}\label{cor: DS for simplicial posets}
	Let $P$ be a graded simplicial poset of rank $d+1$. Then
		\[
			h_{d-j}(P) - h_j(P) = (-1)^j \sum_{t\in P} {d- \rho(t) \choose j} \bigg[\mu_P(t,\hat{1}) - (-1)^{d-1-\rho(t)} \bigg]\quad \text{for } j =0,\dots,d.
		\]
\end{corollary}


\section{Flag Dehn--Sommerville relations}\label{sect:flag DS rel}

The goal of this section is to generalize the Bayer--Billera theorem \cite{MR774533} on flag $h$-vectors of Eulerian balanced simplicial complexes (see also \cite[Cor. 3.16.6]{MR2868112} for the poset version). This result states that if $\Delta$ is an Eulerian balanced simplicial complex of dimension $d-1$, then for all $S\subseteq [d]$, $h_S(\Delta) = h_{[d]-S} (\Delta)$. Recall that the error at face $F\in \Delta$ is defined as
\[\varepsilon_{\Delta}(F) = \tilde{\chi}(
\lk_\Delta F) - (-1)^{d-1-|F|}. \]

The main result of this section is the following. We will provide two proofs.
\begin{theorem}\label{thm: flag DS}
	Let $\Delta$ be a $(d-1)$-dimensional balanced simplicial complex with the coloring map $\kappa:V(\Delta)\to [d]$. Let $S\subseteq[d]$. Then
	\begin{align}\label{eq: flag DS}
	h_S(\Delta) - h_{S^c}(\Delta) = (-1)^{d-|S|} \sum_{F\in\Delta_S}\varepsilon_{\Delta}(F).
	\end{align}
\end{theorem}
Our first proof will rely on the following proposition.

\begin{proposition}\label{prop: short flag vector?} Let $\Delta$ be a $(d-1)$-dimensional balanced simplicial complex with the coloring map $\kappa:V(\Delta)\to [d]$. Let $S$ be a subset of $[d]$ and $i\notin S$, then
	\begin{align}
	\sum_{v:\; \kappa(v) = i} h_S(\lk_\Delta (v))  = h_{S\cup \{i \}}(\Delta) + h_S (\Delta) . 
	\end{align}
\end{proposition}
We delay the proof of Proposition \ref{prop: short flag vector?} until after the proof of Theorem \ref{thm: flag DS}.
\begin{proof}[Proof of Theorem \ref{thm: flag DS}]
	Similar to the second proof of Theorem \ref{DSThm}, we will verify (\ref{eq: flag DS}) by double induction: first on the dimension of $\Delta$, and then on the size of $S$. 
	If $\dim \Delta = 0$, it is easy to check that (\ref{eq: flag DS}) holds. 
	
	Assume $\dim \Delta = d-1 > 0.$ The base case when $S= \emptyset$ follows from Theorem \ref{DSThm}:
	\[h_{\emptyset}(\Delta) - h_{[d]}(\Delta)   =   h_0(\Delta) - h_d(\Delta) \stackrel{\text{Thm.~3.1} }{=} (-1)^d \varepsilon_\Delta(\emptyset).\]
	
	For the inductive step, fix $j\in[d-1]$ and assume that for all $k>j$ and $|S|=d-k$, the equality (\ref{eq: flag DS}) holds.
	Up to reordering of the colors, it suffices to show that 
	\begin{align}\label{eq: NTS}
	h_{[j+1,d]} (\Delta) - h_{[j] }(\Delta) = (-1)^{j} \sum_{F\in\Delta:\; k (F)\subseteq [j+1,d] } \varepsilon_{\Delta}(F) .
	\end{align}
	
By Proposition \ref{prop: short flag vector?},
	\begin{align}\label{eq: LHS}
		&\quad \sum\limits_{\kappa(v)=j+1} \bigg[ h_{[j+1,d]- \{ j+1\}} \bigg(\lk_\Delta (v)\bigg) - h_{[j]} \bigg(\lk_\Delta (v) \bigg) \bigg]  \nonumber\\
		&=\bigg[ h_{[j+1,d]}\bigg(\Delta\bigg) + h_{[j+1,d]- \{ j+1\}} \bigg(\Delta\bigg)\bigg]  - \bigg[ h_{[j+1]}\bigg(\Delta\bigg) + h_{[j]} \bigg(\Delta\bigg)\bigg] \nonumber\\
		& = \bigg[ h_{[j+1,d]}\bigg(\Delta\bigg) -  h_{[j]} \bigg(\Delta\bigg) \bigg]  + \bigg[h_{[j+1,d]- \{ j+1\}} \bigg(\Delta\bigg) - h_{[j+1]}\bigg(\Delta\bigg)\bigg]  \\
		& \stackrel{\text{Ind.~Hyp.}}{=} \bigg[ h_{[j+1,d]}\bigg(\Delta\bigg) -  h_{[j]} \bigg(\Delta\bigg) \bigg]  + \bigg[ (-1)^{j+1} \sum\limits_{\substack{F\in \Delta,\;\\ \kappa(F)\subseteq [j+1,d] - \{j+1 \}} } \varepsilon_\Delta (F)\bigg].\nonumber
	\end{align}
	
	On the other hand, for each vertex $v$ such that $\kappa(v)= j+1$, $\lk_\Delta(v)$ is a  $(d-2)$-dimensional balanced simplicial complex with $\kappa: V\big(\lk_\Delta (v)\big) \to [d] - \{j+1 \}$. By the inductive hypothesis on $\lk_\Delta(v)$, we have
	\begin{align}\label{eq: link}
	h_{[j+1,d] - \{j+1\} } \bigg(\lk_\Delta (v)\bigg) - h_{[j]} \bigg(\lk_\Delta (v)\bigg)  = (-1)^{j} \sum_{\substack{F\in \lk_{\Delta} v, \\ k (F) \subseteq [j+1,d] - \{j+1\}  }  }\varepsilon_{\lk_\Delta (v)} (F).
	\end{align}
	Therefore,
	\begin{align}\label{eq: RHS}
	\sum\limits_{\kappa(v)=j+1} \bigg[ h_{[j+1,d]- \{ j+1\}} \bigg(\lk_\Delta (v)\bigg) - h_{[j]} \bigg(\lk_\Delta (v) \bigg) \bigg]
	&\stackrel{(\ref{eq: link})}{=} \sum\limits_{\kappa(v)=j+1} (-1)^{j} \sum\limits_{\substack{F\in \lk_{\Delta} v, \\ \kappa (F) \subseteq [j+1,d] - \{j+1\}  }  }\varepsilon_{\lk_\Delta (v)} (F)\nonumber\\
	&\stackrel{(\star)}{=}(-1)^{j} \sum\limits_{\substack{\kappa(v) =j+1; \;F\in \lk_{\Delta} v;  \\ k (F\cup v) \subseteq [j+1,d] }  }\varepsilon_{\Delta} (F\cup v). 
	\end{align}
where ($\star$) holds since for any $v\in V(\Delta)$ and $F \in \lk_\Delta (v)$, $\lk_{\lk_\Delta (v)}(F) = \lk_\Delta (F\cup v)$.

	Comparing (\ref{eq: LHS}) and (\ref{eq: RHS}), we obtain
	\begin{align*} h_{[j+1,d]}\bigg(\Delta\bigg) -  h_{[j]} \bigg(\Delta\bigg) &=  (-1)^{j} \bigg[\sum\limits_{\substack{F\in \Delta,\;\\ \kappa(F)\subseteq [j+1,d] - \{j+1 \}} } \varepsilon_\Delta (F) +  \sum\limits_{\substack{\kappa(v) =j+1; \;F\in \lk_{\Delta} v;  \\ k (F\cup v) \subseteq [j+1,d] }  }\varepsilon_{\Delta} (F\cup v) \bigg]\\
	&= (-1)^{j}\sum\limits_{\substack{F\in\Delta, \\  \kappa(F) \subseteq [j+1,d] }} \varepsilon_\Delta (F),
	\end{align*}
	Therefore (\ref{eq: NTS}) holds.
\end{proof}

\begin{proof}[Proof of Proposition \ref{prop: short flag vector?}]
	Recall that $i$ and $S$ are fixed and that $i\notin S$. The proof is a routine computation that relies on the definition of flag $h$-numbers in terms of flag $f$-numbers:
	\begin{align}\label{v notin S}
	\sum_{v: \kappa(v)=i} h_S(\lk_\Delta v)&= \sum_{v: \kappa(v)=i} \sum_{R\subseteq S} (-1)^{|S-R|}f_R(\lk_\Delta v)\nonumber\\
	&= \sum_{R\subseteq S} (-1)^{|S-R|} \sum_{v: \kappa(v)=i} f_R(\lk_\Delta v)\nonumber\\
	&= \sum_{ R\subseteq S}(-1)^{|S\cup \{ i \}|-|R\cup \{ i \}|}  f_{R \cup \{i \}}(\Delta)\nonumber\\
	&= h_{S\cup\{ i \}}(\Delta) - \sum_{T \subseteq S}(-1)^{|S\cup\{ i\} | -|T|}  f_{T}(\Delta)\\
	&= h_{S\cup\{ i \}}(\Delta)  + \sum_{\substack{ T \subseteq S}}(-1)^{|S|-|T|}  f_{T}(\Delta)\nonumber\\
	&= h_{S\cup\{ i \}}(\Delta)  + h_{S}(\Delta).\nonumber
	\end{align}
\end{proof}
Our second proof of Theorem \ref{eq: flag DS} uses the Hilbert series. The idea of this proof is similar to that of \cite[Theorem 3.8]{MR2505437}, which uses the following theorem by Stanley. Recall that $F(\Delta,\lambda)$ is the Hilbert series of $\field[\Delta]$ (w.r.t.~the $\ZZ^d$-grading), that for $F\in\Delta$, $\kappa(F)$ is the set of colors of vertices of $F$, and that for any subset $S\subseteq [d]$, $\lambda^S$ denotes $\prod_{i\in S}\lambda_i$. The following theorem is a corollary of \cite[II, Thm. 7.1]{St-96}.

\begin{theorem}\label{thm: Hilbert series}
	Let $\Delta$ be a $(d-1)$-dimensional balanced simplicial complex and let $F(k [ \Delta] ,\; 1/\lambda)$ be the Hilbert series. Then 
	\[(-1)^d F(k [ \Delta] ,\; 1/\lambda) =  (-1)^{d-1} \tilde{\chi} (\Delta) +  \sum_{F\in\Delta, \; F\neq \emptyset } (-1)^{d-|F|-1}\tilde{\chi} (\lk F)\cdot  \lambda^{\kappa(F)}\cdot \prod\limits_{v\in F} \frac{1}{1 - \lambda_{\kappa(v)} }. \]
\end{theorem}
We will also use the following relation that follows easily from the definition of the flag $h$-numbers, see, the proof of \cite[Theorem 3.8]{MR2505437}: 
\begin{equation} \label{eq:Sw}
\sum_{F\in\Delta, \; F\neq \emptyset } \lambda^{\kappa(F)} \prod_{i\in [d]-\kappa(F)} (1 - \lambda_i )  = \sum_{S\subseteq [d]} \bigg[ h_S(\Delta) - (-1)^{|S|} \bigg] \cdot \lambda^S.
\end{equation}

\begin{proof}[Second proof of Theorem \ref{eq: flag DS} ]
	\begin{align}\label{eq: LHS Hilbeart function}
	&(-1)^d F(k [ \Delta] ,\; 1/\lambda) \nonumber \\ 
	\stackrel{\text{Thm.~\ref{thm: Hilbert series} }}{=} & (-1)^{d-1} \tilde{\chi} (\Delta) +  \sum_{F\in\Delta, \; F\neq \emptyset } (-1)^{d-|F|-1}\tilde{\chi} (\lk F) \cdot  \lambda^{\kappa(F)}\cdot\prod\limits_{v\in F} \frac{1}{1 - \lambda_{\kappa(v)} } \nonumber\\
	=\quad & (-1)^{d-1} \tilde{\chi} (\Delta) +  \bigg(\prod_{j=1}^d \frac{1}{1-\lambda_j}  \bigg) \cdot  \sum_{F\in\Delta, \; F\neq \emptyset } (-1)^{d-|F|-1}\underbrace{\tilde{\chi} (\lk F)}_{ =  (-1)^{d-1-|F|} + \varepsilon_\Delta (F)} \lambda^{\kappa(F)} 
 \prod_{i\in [d]-\kappa(F)} (1 - \lambda_i ) \nonumber\\
		=\quad & (-1)^{d-1} \tilde{\chi} (\Delta) 
	+ \bigg(\prod_{j=1}^d \frac{1}{1-\lambda_j}  \bigg) \cdot  \sum_{F\in\Delta, \; F\neq \emptyset } \lambda^{\kappa(F)} \prod_{i\in [d]-\kappa(F)} (1 - \lambda_i ) \nonumber\\
	& \quad\quad\quad\quad\quad\;\; +  \bigg(\prod_{j=1}^d \frac{1}{1-\lambda_j}  \bigg) \cdot  \sum_{F\in\Delta, \; F\neq \emptyset } (-1)^{d-|F|-1}\cdot \varepsilon_\Delta (F) \lambda^{\kappa(F)}  \prod_{i\in [d]-\kappa(F)} (1 - \lambda_i ) \nonumber\\
	\stackrel{(\ref{eq:Sw})}{=} \quad & (-1)^{d-1} \tilde{\chi} (\Delta) 
	+ \bigg(\prod_{j=1}^d \frac{1}{1-\lambda_j}  \bigg) \cdot \sum_{S\subseteq [d]} \bigg[ h_S(\Delta) - (-1)^{|S|} \bigg] \cdot \lambda^S \\
	& \quad\quad\quad\quad\quad\;\; +  \bigg(\prod_{j=1}^d \frac{1}{1-\lambda_j}  \bigg) \cdot  \sum_{F\in\Delta, \; F\neq \emptyset } (-1)^{d-|F|-1}\cdot \varepsilon_\Delta (F) \lambda^{\kappa(F)} \prod_{i\in [d]-\kappa(F)} (1 - \lambda_i ). \nonumber
	\end{align}
On the other hand, 
	\begin{align}\label{eq: RHS Hilbeart function}
	(-1)^d F(k [ \Delta] ,\; 1/\lambda) = \bigg(\prod_{j=1}^d \frac{1}{1-\lambda_j}  \bigg) \cdot  \sum\limits_{S\subset [d]} h_S(\Delta) \lambda^{[d]-S}.
	\end{align}
Comparing (\ref{eq: LHS Hilbeart function}) with (\ref{eq: RHS Hilbeart function}) and multiplying both sides by $\prod_{j=1}^d(1-\lambda_i)$, we obtain
	\begin{align*}
	(-1)^{d-1} \tilde{\chi}(\Delta)\cdot \prod_{j=1}^d(1-\lambda_i) &+ \sum_{S\subseteq [d]} \bigg[ h_S(\Delta) - (-1)^{|S|} \bigg] \cdot \lambda^S \\
	&+   \sum_{F\in\Delta, \; F\neq \emptyset } (-1)^{d-|F|-1}\cdot \varepsilon_\Delta (F) \lambda^{\kappa(F)}  \cdot \prod_{i\in [d]-\kappa(F)} (1 - \lambda_i )
	= \sum\limits_{S\subset [d]} h_S(\Delta) \lambda^{[d]-S}.
	\end{align*}
	Therefore
	\begin{align*}
	&(-1)^{d-1} \tilde{\chi}(\Delta)\cdot \prod_{j=1}^d(1-\lambda_i) + \sum_{S\subseteq [d]} (-1)^{|S|-1}\cdot \lambda^S +   \sum_{F\in\Delta, \; F\neq \emptyset } (-1)^{d-|F|-1}\cdot \varepsilon_\Delta (F) \lambda^{\kappa(F)}  \cdot \prod_{i\in [d]-\kappa(F)} (1 - \lambda_i ) \\
	=&  \sum\limits_{S\subset [d]} \bigg[h_S(\Delta) \lambda^{[d]-S} - h_S(\Delta) \lambda^{S}\bigg].
	\end{align*}
The coefficient of  $\lambda^S$ on the RHS is $h_{[d]-S} - h_S$. The coefficient of  $\lambda^S$ on the LHS is:
	\begin{align*}&(-1)^{d-|S|-1}\tilde{\chi}(\Delta) + (-1)^{|S|-1} +\sum_{F\in\Delta_S,\;F\neq \emptyset} (-1)^{d-|F|-1} \cdot \varepsilon_\Delta(F)\cdot (-1)^{|S|-|F|}\\
	=& (-1)^{d-|S|-1} \sum_{F\in\Delta_S} \varepsilon_\Delta(F).
	\end{align*}
	Together we obtain 
	\[ h_{[d]-S} - h_S = (-1)^{d-|S|-1} \sum_{F\in\Delta_S} \varepsilon_\Delta(F).\]
\end{proof}

\begin{rem}\label{remark: third pf}
	A similar argument, but using the coarse Hilbert series,  provides yet another proof of Theorem \ref{DSThm}.
\end{rem}

Observe that Theorem \ref{thm: flag DS} refines Theorem \ref{DSThm}: summing eq.~(\ref{eq: flag DS}) over all subsets $S\subseteq [d]$ of size $i$, we obtain
	\begin{align*} 
		\sum_{|S|=i} h_S (\Delta) -\sum_{|S|=i} h_{S^c} (\Delta) = (-1)^{i-1}\sum_{F \in \Delta} {d-|F| \choose i} \varepsilon_{\Delta}(F),
	\end{align*}
	which is equivalent to eq.~(\ref{eqn: simplicial DS formula}).
	
Recall from Section \ref{sect:prelim} that for a graded poset $P$, the complex $O(P)$ is always balanced w.r.t.~the coloring given by the rank function, and that for $S\subseteq [d]$, $\beta_P(S)=h_S(O(P))$. Moreover, by Remark \ref{rmk: chain error = link error}, the chain error of a poset ($\varepsilon_P(-)$) is the same as the link error of its order complex  ($\varepsilon_{O(P)}(-)$). The following corollary now follows directly from Theorem \ref{thm: flag DS}:

\begin{corollary}
	Let $P$ be a graded poset with rank $d+1$ and let $S\subseteq [d]$. Then
	\[\beta_P(S) - \beta_P(S^c)  =  (-1)^{d-|S|}\sum_{C\in \mathcal{C}(P_S)}\varepsilon_P(C) .\]
	where $P_S$ is the $S$-selected subposet of $P$, and $\mathcal{C}(P_S)$ denotes the set of all chains in $P_S \backslash \{\hat{0},\hat{1}\}$.
\end{corollary}


\section{Posets with isolated singularities}\label{sect:1-Sing posets}

Stanley extended the Dehn--Sommerville relations for Eulerian simplicial complexes to the generality of toric $h$-vectors of Eulerian posets. The goal of this and the following sections is to further generalize these relations to more general posets. 

We start by defining the {\bf error function for intervals in posets}. Let $P$ be a graded poset of rank $(d+1)$ and let $[s,t]$ be an interval in $P$. The error of $[s,t]$ is defined as
	\[ 
		e_P ([s,t]) := \mu_P(s,t) - (-1)^{\rho(t) - \rho(s)}.
	\]
From now on we will use $e_P (s,t)$ as the abbreviation for $e_P ([s,t])$.\footnote{We have already defined the error function for chains $\varepsilon_P(C)$, but to study toric $h$-vectors it is easier to use interval errors rather than link errors. The connection between the two will be discussed later in the proof of Corollary \ref{relation of varepsilons or O(P)}.}

\begin{definition} 
	A graded poset $P$ with $\rho(P)=d+1$ has {\bf singularities of rank 1} or  is {\bf $1$-Sing} if all intervals $[s,t]$ in $P$ of length $\rho(t)-\rho(s) \leq d-1$ are Eulerian. 
\end{definition}

\begin{proposition}
	A poset $P$ of rank $\rho(P)=d+1$ is $1$-Sing if and only if its reduced order complex $O(P)$ satisfies the following condition: for all faces $F\in O(P)$ with $dim(F)\geq 1$, $\tilde{\chi}(\lk_{O(P)} F) = (-1)^{d-1-|F|}$.
\end{proposition}

We omit the proof as we will prove a generalization of this result in Proposition \ref{prop: j-Sing and j-CSing}. Our work in the rest of this section is motivated by the following theorem of Stanley \cite{MR951205} and its generalization due to Swartz (see \cite[Theorem 3.15]{MR2505437}). 

\begin{theorem}[Stanley]\label{thm:Stanley}
	Let $P$ be an Eulerian poset of rank $d+1$. Then $\hat{h}_i(P)= \hat{h}_{d-i}(P)$ for all $0\leq i\leq d$.
\end{theorem}

\begin{theorem}[Swartz]\label{thm:Swartz}
	Let $P$ be a semi-Eulerian poset $P$ of rank $d+1$ and let $O(P)$ be its reduced order complex. Then for all $0\leq i\leq d$,
		\[  
			\hat{h}_{d-i}(P) -\hat{h}_i(P)= (-1)^{d-i+1}{d \choose i}[\tilde\chi(O(P)) - (-1)^{d-1}] =(-1)^{d-i+1}{d \choose i}\cdot e_P(\hat{0},\hat{1}) .
		\]
\end{theorem}

\begin{rem} 
	The formula given by Swartz in \cite[Theorem 3.15]{MR2505437} is equivalent to the statement above. Indeed, when $d$ is even, $P$ is Eulerian, and so the right hand-side is zero. If $d$ is odd, then $d-i+1$ and $i$ have the same parity and the formula above agrees with the one in Swartz's Theorem 3.15.
\end{rem}

Using ideas from Swartz's and Stanley's proofs, we establish the following generalization of Theorems \ref{thm:Stanley} and \ref{thm:Swartz} for $1$-Sing posets. For the rest of this section, we let $y=x-1$.

\begin{theorem}\label{1IS-Thm}
	Let $P$ be a graded $1$-{\bf Sing} poset, and let $\rho(P)=d+1$. Then for $i > \floor{\frac{d}{2}}$, 
	
	\begin{align*}
		\hat{h}_{d-i}(P) -\hat{h}_i(P)  
		= (-1)^{d-i+1} 
			\bigg[
			{{d}\choose{i}}e_P(\hat{0},\hat{1}) 
			+{{d}\choose{i}} \sum_{\rho(t)=d} e_P(\hat{0},t) 
			+ {{d-1}\choose{i-1}}\sum_{\rho(s)=1}  e_P(s,\hat{1})
			\bigg]. \nonumber
	\end{align*}
\end{theorem}

\begin{proof}
	The left hand-side is the coefficient of $x^i$ in the polynomial $\hat{h}(P) -  x^d \cdot \hat{h}(P,1/x)
	$. We first prove the following lemma related to this polynomial. From now on we use $\hat{h}(P)$ to abbreviate $\hat{h}(P,x)$ and $\hat{g}(P)$ to abbreviate $\hat{g}(P,x)$. 

\begin{lemma}\label{graded C(Q) lemma}
	Let $P$ be a graded poset with $\rho(P)=d+1$ and let $y=x-1$. Then
		\begin{multline} 
			\hat{h}(P) -  x^d \cdot \hat{h}(P,1/x)= -\big[\mu_P(\hat{0},\hat{1}) - (-1)^{d+1} \big] y^{d} \\
			+ \sum_{\substack{Q=[0,q] \in \tilde{P} \\ 1\leq\rho(Q)\leq d}} \bigg(\bigg[-y^{d-\rho(Q)}\left(\hat{g}(Q) + y\hat{h}(Q)\right)\cdot \mu_P(q,\hat{1}) \bigg] - \bigg[ (-y)^{d-\rho(Q)}\hat{g}(Q,1/x)\cdot x^{\rho(Q)} \bigg] \bigg).\label{eq:2}
		\end{multline}
\end{lemma}

\begin{proof}

By definitions, for $P\neq \mathbbm{1}$,

	\begin{equation}\label{eq:x^d h}
	    x^d \cdot \hat{h}(P,1/x) = \sum_{Q\in \tilde{P},\; Q\neq P} (-y)^{d-\rho(Q)}\hat{g}(Q,1/x)\cdot x^{\rho(Q)},
	\end{equation}
and equivalently,
	\begin{equation}\label{eq:hdef} 
		\hat{h}(P) = \sum_{Q\in \tilde{P},\; Q\neq P} \hat{g}(Q) y^{d-\rho(Q)}. 
	\end{equation}

Multiplying equation (\ref{eq:hdef}) by $y$ and adding $\hat{g}(P)$ to both sides, we obtain that for $P\neq \mathbbm{1}$,
	\[ 
		\hat{g}(P) + y\hat{h}(P) = \sum_{Q\in \tilde{P}} \hat{g}(Q) y^{\rho(P)-\rho(Q)}. 
	\]
Therefore for $P\neq \mathbbm{1}$, 
	\[ 
		y^{-\rho(P)}\cdot \big(\hat{g}(P) + y\hat{h}(P)\big) = \sum_{Q\in \tilde{P}} \hat{g}(Q) y^{-\rho(Q)}. 
	\]
By M\"{o}bius inversion, 
	\begin{equation} \label{eq:1}
		\hat{g}(P) y^{-\rho(P)}  = \mu_P(\hat{0},\hat{1}) +  \sum_{\substack{Q=[0,q] \in \tilde{P} \\ 1\leq\rho(q)\leq d+1}}y^{-\rho(q)}\cdot \left(\hat{g}(Q) + y\hat{h}(Q)\right)\cdot \mu_P(q,\hat{1}).
	\end{equation} 
Multiplying (\ref{eq:1}) by $y^{\rho(P)}$ and then subtracting $\hat{g}(P) +y\hat{h}(P) $ yields
	\[
		-y\hat{h}(P)   = \mu_P(\hat{0},\hat{1})y^{\rho(P)} +\sum_{\substack{Q=[0,q] \in \tilde{P} \\ 1\leq\rho(q)\leq d}}y^{{\rho(P)}-\rho(q)}\left(\hat{g}(Q) + y\hat{h}(Q)\right)\cdot \mu_P(q,\hat{1}),
	\]
and so 
	\[
		\hat{h}(P)   = -\mu_P(\hat{0},\hat{1})y^{d} -\sum_{\substack{Q=[0,q] \in \tilde{P} \\ 1\leq\rho(q)\leq d}}y^{d-\rho(q)}\left(\hat{g}(Q) + y\hat{h}(Q)\right)\cdot \mu_P(q,\hat{1}).
	\]
This, together with equation ($\ref{eq:x^d h}$), proves the lemma. \end{proof}

Next we prove the following lemma, which helps us further simplify equation (\ref{eq:2}) for semi-Eulerian posets.

\begin{lemma}\label{1-IS C(Q) lemma}
	Let $Q$ be a semi-Eulerian poset with $\rho(Q)=r+1$, let $s = \floor*{\frac{r}{2}}$, and let $y=x-1$. Then
		\[
			\hat{g}(Q) +y\hat{h}(Q) = x^{\rho(Q)} \hat{g}(Q,1/x) + \sideset{}{^*}\sum_{k=r-s+1}^{r+1} \underbrace{(-1)^{r-k} {{r+1}\choose{k}} e_Q(Q)}_{\gamma_{k}}\cdot x^{k}, 
		\]
	where $\sideset{}{^*}\sum$ means that, if r is odd, then there is an extra summand, $\frac{1}{2}\gamma_k x^{k}$, for $k = r-s$.
\end{lemma}

\begin{proof}
Recall that $\hat{h}(Q) = \hat{h}_r + \hat{h}_{r-1}x +\dots +\hat{h}_0 x^r$. This together with the definition of $\hat{g}(Q)$ implies 
	\[
		\hat{g}(Q) +y\hat{h}(Q) = (\hat{h}_{r-s} - \hat{h}_{r-s-1})x^{s+1} + (\hat{h}_{r-s-1} - \hat{h}_{r-s-2})x^{s+2} + \dots + (\hat{h}_1 - \hat{h}_0)x^r +\hat{h}_0 x^{r+1},  
	\]
while
	\[
		x^{\rho(Q)} \hat{g}(Q,1/x) = (\hat{h}_{r-s} - \hat{h}_{r-s+1})x^{r-s+1} + (\hat{h}_{r-s+1} - \hat{h}_{r-s+2})x^{r-s+2}+\dots + (\hat{h}_{r-1} - \hat{h}_{r})x^r +\hat{h}_{r}x^{r+1}. 
	\]
By Theorem \ref{thm:Swartz}, if $Q$ is semi-Eulerian, then $\hat{h}_{r-k} = \hat{h}_{k} + (-1)^{r-k+1} {{r}\choose{k}} e_Q(\hat{0},\hat{1})$. Hence for $k<r$,
	\[
		\hat{h}_{r-k} - \hat{h}_{r-k-1} = (\hat{h}_{k} -\hat{h}_{k+1}) + (-1)^{r-k+1}\cdot \bigg[{{r}\choose{k}}+{{r}\choose{k+1}} \bigg]  e_Q(\hat{0},\hat{1}),
	\]
and since ${{r}\choose{k}}+{{r}\choose{k+1}} = {{r+1}\choose{k+1}}$, we infer that
	\[
		\hat{h}_{r-k} - \hat{h}_{r-k-1} = (\hat{h}_{k} -\hat{h}_{k+1}) + \underbrace{(-1)^{r-k+1} {{r+1}\choose{k+1}} \cdot e_Q(\hat{0},\hat{1}) }_{\gamma_{k+1}(Q)}.
	\]
Comparing the coefficients of $x^{k+1}$ in $\hat{g}(Q) +y\hat{h}(Q)$ and $ x^{\rho(Q)} \hat{g}(Q,1/x)$, yields the lemma. 
\end{proof}

Now we resume the proof of Theorem \ref{1IS-Thm}. For any lower interval $Q=[\hat{0},q]$ in $P$, if $1< \rho(q) < d$, then $Q$ is Eulerian. If $\rho(q)=d$, then $Q$ is semi-Eulerian. By Swartz's result and the lemma above,
	\[ \hat{g}(Q) + y\hat{h}(Q) =  
		\begin{cases} 
		       x^{\rho(q)}\hat{g}(Q,1/x) & \text{ if } 1\leq \rho(q) < d \\
		      x^{\rho(q)} \hat{g}(Q,1/x) +\sideset{}{^*}\sum\limits_{i=\floor*{\frac{d}{2}}+1}^{d} \gamma_i(Q)\cdot x^{i} & \text{ if } \rho(q) = d.
	   \end{cases}
	\]

We can now simplify equation ($\ref{eq:2}$):
	\begin{multline} 
		\hat{h}(P) -  x^d \cdot \hat{h}(P,1/x)= -e_P(\hat{0},\hat{1}) y^{d} \\
		- \sum_{\rho(Q) =1} y^{d-1}x\big[\mu_P(q,\hat{1}) -(-1)^{d}\big] 
		+\sum_{\rho(Q)=d} \;\;\sideset{}{^*}\sum\limits_{i=\floor{\frac{d}{2}}+1}^{d} \gamma_i(Q)\cdot x^{i},\label{eq:3}
	\end{multline}
and when $d$ is even, the last summation $\sideset{}{^*}\sum$ on the right hand-side has an extra summand $\frac{1}{2} \gamma_i x^i$ for $i=\floor{\frac{d}{2}}$. Comparing like-terms from both sides: for $i >\floor*{\frac{d}{2}}$,

	\begin{align}\label{eqn: like terms}
		\begin{split}
			\hat{h}_{d-i}&  -\hat{h}_i = (-1)^{d-i+1} \bigg[ {{d}\choose{i}}e_P(\hat{0},\hat{1}) +    {{d-1}\choose{i-1}} \sum_{\rho(q)=1}e_P(q,\hat{1})
			+{{d}\choose{i}}\sum_{\rho(q)=d}  e_P(\hat{0},q)\bigg]
		\end{split}
	\end{align}
as desired.
\end{proof}

The following special case is worth mentioning:  if $d$ is even and $i = \floor*{\frac{d}{2}} = \frac{d}{2} = d-i$, the left-hand side of (\ref{eqn: like terms}) is simply zero, and hence so is the right-hand side. This observation leads to the following Corollary. We will generalize it later in Corollaries \ref{relations of errors for P} and \ref{relation of varepsilons or O(P)}.

\begin{corollary}\label{Cor: 1-Sing chi relations}
	Let $P$ be a $1$-Sing poset with odd rank $\rho(P)=d+1$, and let $O(P)$ be the reduced order complex of $P$. Then
		\[
			2 (\tilde{\chi}(O(P)) +1) = \sum_{\rho(q)=1 \text{ or }d} 1 - \sum_{\rho(q)=1 \text{ or } d} \tilde{\chi}\big(\lk_{O(P)}v_q \big) 
		\]
	where $v_q$ is the vertex in $O(P)$ that corresponds to the element $q\in P$.
\end{corollary}

\begin{proof}
  If $d$ is even, and $ i = \floor{\frac{d}{2}}$, then $d-i =i$. Equation (\ref{eqn: like terms}) gives: 
	   \begin{multline}
		   0 = \hat{h}_{\frac{d}{2}} - \hat{h}_{\frac{d}{2}} = (-1)^{\frac{d}{2}+1} {d \choose \frac{d}{2}} \big[ \mu_P(\hat{0},\hat{1}) -(-1) \big]
		  \\
		   + \sum_{\rho(q)=1} (-1)^{\frac{d}{2}+1} {d-1 \choose \frac{d}{2} -1} \big[ \mu_P(q,\hat{1}) -1 \big]+ \frac{1}{2} \sum_{\rho(q)=d} (-1)^{\frac{d}{2} -1} {d \choose \frac{d}{2}} \big[ \mu_P(\hat{0},q) -1 \big].
	   \end{multline}
Since $ {d-1 \choose \frac{d}{2} -1} = \frac{1}{2} {d \choose \frac{d}{2}}$, this can be simplified to
  	\[
  		0 = \big[ \mu_P(\hat{0},\hat{1}) +1 \big] + \frac{1}{2}  \sum_{\rho(q)=1} \big[\mu_P(q,\hat{1}) - 1   \big]+ \frac{1}{2} \sum_{\rho(q)=d} \big[\mu_P(\hat{0},q) - 1   \big]  .  
  	\]
Hence
	\begin{multline} 
		  2 (\mu_P(\hat{0},\hat{1}) +1) =  \sum_{\rho(q)=1} \big[1- \mu_P(q,\hat{1}) \big] +\sum_{\rho(q)=d} \big[1 -\mu_P(\hat{0},q)   \big]  \\
		  =\sum_{\rho(q)=1, d} 1 + \sum_{\rho(q)=1, d} \mu_P(\hat{0},q)\cdot \mu_P(q,\hat{1}) = \sum_{\rho(q)=1, d} 1 - \sum_{\rho(q)=1, d} \tilde{\chi}\big(\lk_{O(P)}F_{\{q\}}\big).\nonumber
	\end{multline}
\end{proof}

\begin{rem}
	The above corollary is a generalization of a result from \cite{MR2943717} asserting that for a $(d-1)$-dimensional simplicial psuedomanifold $\Delta$ with isolated singularities,
		\[ 
			2(\tilde{\chi}(\Delta) +1) = |V| - \sum_{v\in V} \tilde{\chi}(\lk_\Delta v). 
		\]
\end{rem}


\section{Posets with singularities of higher degrees}\label{section: j-Sing}
This section generalizes several results from Section \ref{sect:1-Sing posets}, most notably Theorem \ref{1IS-Thm}.

\subsection{$j$-Sing Posets}
We start with a definition of posets with singularities of degree at most $j$. A similar notion was introduced in \cite{MR1867234}.

\begin{definition}\label{Def j-Sing}
	For a finite graded poset $P$ of rank $\rho(P)=d+1$, we recursively define the notion of $j$-{\bf Sing}:
	   \begin{itemize}
		    \item[] 
		    \subitem $P$ is {\bf$(-1)$-Sing} if $P$ is Eulerian.
		    \subitem $P$ is {\bf $0$-Sing} if $P$ is semi-Eulerian.
		    \subitem $P$ is {\bf $j$-Sing} if every interval of length $\leq d$ in $P$ is {\bf $(j-1)$-Sing}.
	\end{itemize}
\end{definition}

A few remarks are in order.
\begin{rem}\label{remark: Def j-Sing}
	A poset $P$ is $j$-Sing if and only if for all $s\leq j+1$, every interval of length $\leq d+1-s$ in $P$ is $(j-s)$-Sing. Also, $P$ is $j$-Sing if and only if every interval of length $\leq d-j$ in $P$ is Eulerian, i.e., every such interval $[s,t]$ has $e_P(s,t)=0$.
\end{rem}

\begin{rem}\label{remark: odd j even d}
	When $j$ is odd (even, resp.), all $j$-Sing posets P with even (odd, resp.) rank are in fact $(j-1)$-Sing. This follows from Definition \ref{Def j-Sing} and the fact that every semi-Eulerian poset of odd rank is actually Eulerian.
\end{rem}

\begin{definition}
	A $(d-1)$-dimensional pure simplicial complex $\Delta$ is called a $j$-{\bf singular complex} if $\varepsilon_{\Delta}(F)=0$ for every face $F\in \Delta$ of $\dim(F)\geq j$, i.e., $\tilde{\chi}(\lk_\Delta F ) = (-1)^{d-1-|F|}$.
\end{definition}

\begin{proposition}\label{prop: j-Sing and j-CSing} 
	The following are equivalent:
	\begin{enumerate}
		\item[(1).] A poset $P$ is $j$-Sing.
		\item[(2).] The order complex of $P$, $O(P)$, is a $j$-singular complex.
		\item[(3).] For every chain $C = \{\hat{0}=t_0 <t_1 <\dots <t_{i-1}<t_i=\hat{1}  \}$ in $P$ such that $i > j+1$,
		the (chain) error $\varepsilon_P(C)=0$.
	\end{enumerate}
\end{proposition}

\begin{proof} Assume $P$ has rank $d+1$, and so $\dim O(P)=d-1$. 
\begin{enumerate}
	\item[(2)$\Longleftrightarrow$(3)] 
		This is clear since by definitions, $\varepsilon_P(C) =\varepsilon_{O(P)}(F_C)$, and every chain $C$ in (3) corresponds to a face $F_C\in O(P)$ of dimension $\geq j$.
	\item[(1)$\Longrightarrow$(3)] 
		$P$ is $j$-Sing if and only if every interval of length $\leq d-j$ in $P$ is Eulerian. Therefore, for any interval $[s,t]$ in $P$ with $\rho(t)-\rho(s)\leq d-j$, 
		$\mu_P(s,t)=(-1)^{\rho(t)-\rho(s)}$. For any chain $C$ as in (3), 
			\[
				\rho(t_k) - \rho(t_{k-1})\leq d-j \quad \text{ for all } 1\leq k\leq i,
			\] 
		therefore every interval $[t_{k-1},\;t_k]$ is Eulerian. This implies $\mu_P(C)=(-1)^{d+1}$ and so $\varepsilon_P(C) =0$.

	\item[(3)$\Longrightarrow$(1)] 
		Pick any interval $[s,t]$ in P with $\rho(t)-\rho(s) \leq d-j$. Consider a maximal chain $(\hat{0}=t_0<t_1<\cdots<t_{\rho(s)}=s)$ in $[\hat{0},s]$ and a maximal chain $(t=t_{\rho(s)+1}<t_{\rho(s)+2}<\cdots<t_k=\hat{1})$ in $[t,\hat{1}]$, and let C be the union of these two chains: 	
			\[
				C=(\hat{0}=t_0<t_1<\dots<s<t<\dots t_{k-1}<t_k=\hat{1}).
			\]
		In particular, $\rho(t_i)-\rho(t_{i-1})=1$ unless $i=\rho(s)+1$. Then C has length $k\geq j$, and so by (3),
			\[
				\mu_P(\hat{0},t_1) \mu_P(t_1,t_2)\dots \mu_P(s,t) \dots\mu_P(t_{i-1},\hat{1}) = (-1)^{d+1}.
			\]
		Since intervals of length $1$ always have the M\"{o}bius value $-1$, this forces $\mu_P(s,t)= (-1)^{\rho(t)-\rho(s)}$. All intervals of length $\leq d-j$ are Eulerian, therefore $P$ is $j$-Sing.
	\end{enumerate}
\end{proof}
To make the exposition cleaner, we introduce the following notation: 

\begin{definition}
For any $j$-Sing poset $P$ of rank $\rho(P)=d+1$, define
	\[ 
		A^{(j)}_k (P):= \hat{h}_{d-k}(P) - \hat{h}_{k}(P).
	\]

\end{definition}

Note that any graded poset of rank $d+1$ is automatically $(d-1)$-Sing. The following claim on $A^{(j)}_0(P)$ can be shown by an easy induction on $j$.

\begin{proposition}\label{prop:A_0}
	For any graded poset $P$ of rank $d+1$, 
		\[
			A^{(j)}_0 (P)= (-1)^d \cdot e_P(\hat{0},\hat{1}).
		\]
\end{proposition}

\subsection{Dehn--Sommerville relations}
We are now in a position to generalize Theorem \ref{1IS-Thm}. First, by comparing the polynomials $\hat{g}(Q) + (x-1)\hat{h}(Q)$ and $x^{\rho(Q)}\hat{g}(Q,\frac{1}{x})$ (as in the proof of Lemma \ref{1-IS C(Q) lemma}), we obtain the following extension of Lemma \ref{1-IS C(Q) lemma}. We omit the proof. \\

\begin{lemma}\label{generalized lemma} 
	Let $Q$ be a $j$-Sing poset with rank $r+1$, and let $y= x-1$. Then
		\[ 
			\hat{g}(Q) + y\hat{h}(Q)= x^{\rho(Q)}\hat{g}(Q,\frac{1}{x}) + \sideset{}{^*}\sum_{k=\floor*{\frac{r+1}{2}} +1}^{r+1} \bigg[A^{(j)}_{k-1} (Q) -A^{(j)}_{k} (Q) \bigg] x^{k},
		\]
	where $\sideset{}{^*}\sum$ means that, if r is odd, then there is an extra summand, $\frac{1}{2}[A^{(j)}_{k-1} (Q) -A^{(j)}_{k} (Q)]x^{k}$ for $k = \floor*{\frac{r+1}{2}}$, which equals $-A^{(j)}_{\floor*{\frac{r+1}{2}}} (Q)\cdot x^{\floor*{\frac{r+1}{2}}}$.
\end{lemma}

The first main result of this section is the following generalization of Theorem \ref{1IS-Thm}:

\begin{theorem}\label{generalized Thm}
	Let $P$ be a $j$-Sing poset with rank $d+1$, where $-1 \leq j \leq d$. For $q\in P$, denote by $Q$ the interval $[\hat{0},q]$. Then
		\begin{multline}\label{eqn:j-Sing}
		\hat{h}(P) - x^d \hat{h}(P,\frac{1}{x}) = -\sum_{\rho(q)\leq j} \hat{g}(Q,\frac{1}{x}) \cdot e_P(q,\hat{1}) \cdot y^{d-\rho(q)}\cdot x^{\rho(q)} \\
		  -\sum_{d-j<\rho(q)\leq d} \quad \sideset{}{^*}\sum\limits_{k=\floor{\frac{\rho(q)}{2}}+1}^{\rho(q)} \bigg(A^{(j+\rho(q)-d-1)}_{k-1} (Q) -A^{(j+\rho(q)-d-1)}_{k} (Q) \bigg)\cdot \mu(q,\hat{1})\cdot y^{d-\rho(q)}x^{k}.
		\end{multline}
\end{theorem}

Before proving this formula, we first show that for small $j$'s, equation (\ref{eqn:j-Sing}) reduces to previous results.

\begin{example}\leavevmode
	\begin{enumerate}
		\item If $P$ is $(-1)$-Sing, then $P$ is Eulerian. Hence by \cite{MR951205}, $A^{(-1)}_k(P)=0$ for every $k$. This matches equation (\ref{eqn:j-Sing}) as both sums on the right hand-side of (\ref{eqn:j-Sing}) are empty.
		\item If $P$ is $0$-Sing, then $P$ is semi-Eulerian, so that all proper intervals $Q \in \tilde{P}$ are $(-1)$-Sing (i.e., Eulerian). In this case, by \cite{MR2505437}, 
			\[ 
				A^{(0)}_k (P) =  (-1)^{d-k+1} {d \choose k} e_P(\hat{0},\hat{1}). 
			\]
		This coincides with equation (\ref{eqn:j-Sing}) as the summand corresponding to $Q$ of rank $0$ is the only term showing up on the right hand-side of (\ref{eqn:j-Sing}) for $j=0$.

		\item If  $P$ is $1$-Sing, then equation (\ref{eqn:j-Sing}) implies that
		\footnotesize 
			\begin{align}
			\begin{split}\label{formula:1-IS}
				A^{(1)}_k (P) =  (-1)^{d-k+1} {d \choose k} e_P(\hat{0},\hat{1})
			+ \sum_{\rho(Q)=1} (-1)^{d-k+1} {d-1 \choose k-1} e_P(q,\hat{1}) + \sum_{\rho(Q)=d} \bigg(  A^{(0)}_{k-1} (Q) -A^{(0)}_k (Q) \bigg)\\
			= (-1)^{d-k+1} {d \choose k} e_P(\hat{0},\hat{1}) + \sum_{\rho(q)=1} (-1)^{d-k+1} {d-1 \choose k-1} e_P(q,\hat{1})
			+ \sum_{\rho(q)=d} (-1)^{d-k+1} {{d}\choose{k}} e_P(\hat{0},q), 
			\end{split}
			\end{align}
		\normalsize
where for the last step we used that, if $\rho(q)=d$, then the interval $[\hat{0},q]$ is semi-Eulerian. This agrees with our formula in Theorem \ref{1IS-Thm}.
	\end{enumerate}
\end{example}

\begin{proof}[Proof of Theorem \ref{generalized Thm}]
 By Lemma \ref{graded C(Q) lemma}, the following equation holds for an arbitrary graded poset $P$ of rank $d+1$.
		\begin{multline}
			\hat{h}(P,x) - x^d \hat{h} (P, \frac{1}{x}) 
			= - y^{d}e_P(\hat{0},\hat{1})\nonumber\\
			+ \sum_{\substack{Q=[\hat{0},q] \in \tilde{P} \\1\leq \rho(Q)\leq d}} \underbrace{ \bigg(\bigg[ -y^{d-\rho(Q)}(\hat{g}(Q) + y\hat{h}(Q))\mu_P(q,\hat{1})\bigg]
			- \bigg[ (-y)^{d-\rho(Q)} \hat{g}(Q,\frac{1}{x})x^{\rho(Q)}  \bigg] \bigg)}_{C(Q)}.  \nonumber
		\end{multline}
From now on we assume $P$ to be $j$-Sing of rank $d+1$. We are interested in $A^{(j)}_k (P)$, which is the coefficient of $x^k$ in $\hat{h}(P,x) - x^d \hat{h} (P, \frac{1}{x})$.
Since $P$ is $j$-Sing, Remark \ref{remark: Def j-Sing} implies that each interval $Q=[\hat{0},q]\subseteq P$ of rank $r+1$ is $(j+r-d)$-Sing. In particular,
	\[
		\mu(q,\hat{1})=(-1)^{d-\rho(q)+1} \quad\quad \text{ for }\rho(q)> j.
	\]

Using this observation together with Lemma \ref{generalized lemma}, we conclude that the following holds for any interval $Q$ of rank $r+1$:\\
Case 1: $j<\floor{\frac{d}{2}}\leq d-j$.
	{\small
		\[
		C(Q) = 
			\begin{cases}
			  -\hat{g}(Q,\frac{1}{x}) \cdot e_P(q,\hat{1})\cdot y^{d-\rho(Q)}x^{\rho(Q)} & \text{ for }\rho(Q)\leq j \\
			  \\
			  0 &  \text{ for }\rho(Q) \in (j, d-j] \\
			  \\
			  \sideset{}{^*}\sum\limits_{k=\floor{\frac{r+1}{2}}+1}^{r+1} \bigg(A^{(j+r-d)}_{k-1} (Q) -A^{(j+r-d)}_{k} (Q) \bigg)x^{k}\cdot(-y)^{d-\rho(Q)} &  \text{ for } \rho(Q)> d-j.
			\end{cases}
		\]
	}
Case 2: $j\geq \floor{\frac{d}{2}}$.
	{\small 
		\[
		C(Q) = \begin{cases}
		  -\hat{g}(Q,\frac{1}{x}) \cdot e_P(q,\hat{1})\cdot y^{d-\rho(Q)}x^{\rho(Q)}, & \text{ for }\rho(Q)\leq d-j \\
		  \\
		  \\
		  
		  -\hat{g}(Q,\frac{1}{x}) \cdot e_P(q,\hat{1})\cdot y^{d-\rho(Q)}x^{\rho(Q)} \\
		  \quad\quad\quad -\sideset{}{^*}\sum\limits_{k=\floor{\frac{r+1}{2}}+1}^{r+1} \bigg(A^{(j+r-d)}_{k-1} (Q) -A^{(j+r-d)}_{k} (Q) \bigg)\cdot \mu(q,\hat{1})y^{d-\rho(Q)} x^{k}, &  \text{ for }\rho(Q)\in (d-j, j] \\
		  \\
		  \\
		  - \sideset{}{^*}\sum\limits_{k=\floor{\frac{r+1}{2}}+1}^{r+1} \bigg(A^{(j+r-d)}_{k-1} (Q) -A^{(j+r-d)}_{k} (Q) \bigg)\cdot \underbrace{\mu(q,\hat{1})}_{\quad=(-1)^{d-r}}\cdot y^{d-\rho(Q)}x^{k}, &  \text{ for } \rho(Q)\geq j+1.
		\end{cases}
		\]
	}
In both cases, comparing the coefficients on both sides, yields the statement.
\end{proof}

\noindent This shows that the difference between $\hat{h}_{d-k}(P)$ and $\hat{h}_{k}(P)$ is a ``weighted" sum of the error functions of the intervals in $P$. Unfortunately, as $j$ gets larger, the length of our formula expands very quickly. In the rest of this section, we will simplify this formula for $j>\frac{d}{2}$ and $k>\frac{d+j}{2}$, in Theorem \ref{main gen Thm}. Our main tool is the following result, that might be of interest on its own.

\begin{corollary}\label{relations of errors for P}
	Let $P$ be a $j$-Sing poset of rank $d+1$.

	$\bullet$ If $d$ is even, then
		\[ 
			2 e_P(\hat{0},\hat{1}) = - \sum_{1\leq \rho(t)\leq j}e_P(t,\hat{1}) - \sum_{d-j+1\leq \rho(t)\leq d} e_P(\hat{0},t).
		\]

	$\bullet$ If $d$ is odd, then
		\[
			\sum_{1\leq \rho(t)\leq j} e_P(t,\hat{1}) = \sum_{d-j+1\leq \rho(t)\leq d} e_P(\hat{0},t). 
		\]
\end{corollary}

\begin{proof}

We will only treat the case of $j\geq \lfloor \frac{d}{2}\rfloor$ since the case of $j <\lfloor \frac{d}{2}\rfloor$ is very similar.

By Theorem \ref{generalized Thm},
\footnotesize 
	\begin{align}\label{relation of varepsilons}
		\begin{split}
		\hat{h}_0(P) &-  \hat{h}_d(P) \\
		&= \sum_{0\leq \rho(t)\leq j} (-1)^{\rho(t)+1}\cdot \mu(\hat{0},t) \cdot e_P(t,\hat{1}) + \sum_{d-j+1\leq \rho(t)\leq d} (-1)^{d-\rho(t)}\cdot \big[\hat{h}_0([\hat{0},t]) -\hat{h}_{\rho(t)-1}([\hat{0},t]) \big]\\
		&\stackrel{\star}{=} \sum_{0\leq \rho(t)\leq j} (-1)^{\rho(t)+1}\cdot \mu(\hat{0},t) \cdot e_P(t,\hat{1}) + \sum_{d-j+1\leq \rho(t)\leq d} (-1)^{d-\rho(t)}\cdot (-1)^{\rho(t)+1} e_P(\hat{0},t)\\
		&= -\sum_{0\leq \rho(t)\leq d-j} e_P(t,\hat{1})+ \sum_{d-j+1\leq \rho(t)\leq j} [(-1)^{d+1}\mu(\hat{0},t) -\mu(t,\hat{1})] + \sum_{j+1\leq \rho(t)\leq d}(-1)^{d+1} e_P(\hat{0},t)\\
		&= -\sum_{0\leq \rho(t)\leq j} e_P(t,\hat{1}) + \sum_{d-j+1\leq \rho(t)\leq d}(-1)^{d+1} \cdot e_P(\hat{0},t) 
		\end{split}
	\end{align} 
\normalsize
where the equality ``$\stackrel{\star}{=}$" follows from Proposition \ref{prop:A_0}.
However by Proposition \ref{prop:A_0}, the left hand-side should also equal $(-1)^d e_P(\hat{0},\hat{1})$. Comparing it with the last line of (\ref{relation of varepsilons}) yields the result. 
\end{proof}

If $j<\floor{d/2}$, Corollary \ref{relations of errors for P} is equivalent to the following geometric interpretation that generalizes Corollary \ref{Cor: 1-Sing chi relations}. 

\begin{corollary}\label{relation of varepsilons or O(P)}
	Let $P$ be a $j$-Sing poset of rank $d+1$ with $j<\floor{d/2}$. Let $\varepsilon_{O(P)}$ be the error function associated to $O(P)$.
	
	$\bullet$ If $d$ is even, then
		\[ 
			2 \varepsilon_{O(P)}(\emptyset) = - \sum_{F \in O(P), |F|\leq j} \varepsilon_{O(P)}( F).
		\]
	
	$\bullet$ If $d$ is odd, and for every face $F\in O(P)$ we let $F^{top} $ and $F_{bot}$ denote the top and the bottom element of $F$ (viewing $F$ as a chain in $P\setminus \{\hat{0},\;\hat{1}\}$), then
		\[
			\sum_{\substack{F \in O(P) \\\rho(F^{top})\leq j} }\varepsilon_{O(P)}(F) = \sum_{\substack{F \in O(P) \\\rho(F_{bot})\geq  d-j+1} }\varepsilon_{O(P)}(F) . 
		\]
\end{corollary}

\begin{proof}
	First notice that when $j<\floor{d/2} \leq d-j$, any face $F$ with $\rho(F^{top})>j$ or $\rho(F_{bot})< d-j+1$ will have $\varepsilon_{O(P)} (F)=0$ since all the intervals defined by this chain are Eulerian. This means
		\[ 
			\sum_{F \in O(P), |F|\leq j} \varepsilon_{O(P)}( F)=   \sum_{\substack{F \in O(P) \\\rho(F^{top})\leq j} }\varepsilon_{O(P)}(F) + \sum_{\substack{F \in O(P) \\\rho(F_{bot})\geq  d-j+1} }\varepsilon_{O(P)}(F) . 
		\]
	
	Now it suffices to show the following claim: for $q,\;t\in P$, $\rho(q)\leq j$ and $\rho(t) \geq d-j+1$,
		\[
			\sum_{\substack{F \in O(P) \\ F^{top}=q} } \varepsilon_{O(P)} (F)  = e_P(q,\hat{1}) \quad\quad\quad\text{and}\quad\quad\quad \sum_{\substack{F \in O(P) \\ F_{bot}=t} } \varepsilon_{O(P)} (F)  = e_P(\hat{0},t).
		\]
	Recall that every face $F$ in $O(P)$ corresponds to a chain in $P\setminus \{\hat{0},\hat{1}\}$, therefore by abusing notation, we can write $F= \{t_1 < t_2 <\dots <t_k =  q\}$. If $\rho(q)\leq j$, 
	then all intervals $[t_i,t_{i+1}]$ are Eulerian, and so
		\[
			\tilde{\chi}( \lk_{O(P)} F) = (-1)^{|F|}\cdot \mu(\hat{0},t_1)\mu(t_1, t_2) \dots \mu(q,\hat{1}) = (-1)^{|F|+\rho(q)}\cdot \mu(q,\hat{1}).  
		\]
	Hence $\varepsilon_{O(P)} (F) = (-1)^{|F|+\rho(q)} e_P(q,\hat{1})$. In particular, if $F$ is a facet in $O([\hat{0},q])$, i.e., a saturated chain in $[\hat{0},q]$, then $\rho(q)=|F|$ and $\varepsilon_{O(P)} (F) = e_P(q,\hat{1})$. If $F$ a ridge in $O([\hat{0},q])$, then $\varepsilon_{O(P)} (F) = - e_P(q,\hat{1})$, etc. This, together with the fact that $O([\hat{0},q])$ is Eulerian, implies
			\begin{equation*} 
				\sum_{\substack{F \in O(P) \\ F^{top}=q} } \varepsilon_{O(P)} (F)  = (-1)^{\rho(q)} \cdot \tilde{\chi}(O\left([\hat{0},q])\right)\cdot e_P(q,\hat{1}) = e_P(q,\hat{1}).
			\end{equation*}
A symmetric argument takes care of the other half of the claim. Our statement now follows from Corollary \ref{relations of errors for P}	
\end{proof}

\begin{rem}
	The polynomial in (\ref{eqn:j-Sing}) is a symmetric polynomial with half of its coefficients negated. In particular, it is the following (here $A_k$ stands for $A_k(P)$):
		\[
			A_d x^d + A_{d-1}x^{d-1}+\dots  + A_{k}x^{k}   +\dots -A_{k} x^{d-k}- \dots- A_{d-1}x -A_d.
		\]
	Therefore the coefficients of $x^k$ and $x^{d-k}$ add up to zero for every $k\neq \frac{d}{2}$. The case of $k=d$ was used to obtain Corollaries \ref{relations of errors for P} and \ref{relation of varepsilons or O(P)}. A natural open problem is the following: can we make use of other equalities arising from comparing the coefficients of $x^k$ and $x^{d-k}$ for $k\neq d$ of this polynomial? 
\end{rem}

The next main theorem in this section will give an explicit formula for $A_k^{(j)}(P)$. Before stating the theorem, we first introduce some notation. Given a poset $T=[\hat{0},t]$ and integers $u$ and $v$, we define the following function.
	\[
		C(T,u,v):= {u-\rho(t) \choose v} - \hat{g}_{\rho(t)-2}(T)\cdot {u-\rho(t) \choose v-1} +\dots  +(-1)^{m} \hat{g}_{\rho(t)-1-m}(T)\cdot {u-\rho(t) \choose v-m }
	\]
where $ m= \floor{(\rho(t)-1)/2}$.

Note that by Pascal's rule on binomial coefficients, for each $u,v$, we have 
	\begin{align}\label{eq:recurrence} 
		C(T,u,v) +C(T,u,v+1) = C(T,u+1,v+1).
	\end{align}

\begin{theorem}\label{main gen Thm} 
	Let $P$ be a $j$-Sing poset of rank $d+1$ (with $d>2j$), and let $k>(d+j)/2$. For each element $t\in P$, let $T:=[\hat{0},t]$, then
		\begin{align}\label{eqn:simplified}
			A^{(j)}_k (P) = (-1)^{k}\sum_{\substack{\rho(t)\leq j }} e_P(t,\hat{1})\cdot C(T,d,k) \quad\quad \text{for }k > (d+j)/2.
		\end{align} 
\end{theorem}

\begin{proof}
The proof is by induction on j. The base cases of $j=-1$ and $0$ are immediate from Stanley's and Swartz's results, see Theorems \ref{thm:Stanley} and \ref{thm:Swartz}. The $j=1$ case can be easily checked using (\ref{formula:1-IS}) and Corollary \ref{relations of errors for P}. The inductive hypothesis is that for all $j'<j$, all $j'$-Sing posets $Q$ of rank $d'+1$, and all $k>\frac{d'+j'}{2}$,
	\begin{align}\label{inductive hyp} 
		A_k^{(j')}(Q) =  (-1)^{k}\sum_{\substack{\rho(t)\leq j'}} e_P(t,\hat{1})\cdot C(T,d',k). 
	\end{align}

If $P$ is a $j$-Sing poset and $k > \frac{d+j}{2}$, then since $A^{(j)}_k(P)$ is the coefficient of $x^k$ in Equation (\ref{eqn:j-Sing}), 
\footnotesize
	\begin{align}\label{eqn:second half of 6.9} 
	\begin{split}
	A^{(j)}_k(P) &= (-1)^{d-k+1} \sum_{\substack{\rho(t)\leq j }}e(t,\hat{1}) \cdot C(T,d,d-k) \\
	 &-\sum_{\substack{\rho(q)=d-j+b\\ 0<b\leq j\\ k-(d-\rho(q))-1 \leq a\leq k-1}}  (-1)^{k-a}\cdot {d-\rho(q) \choose {k-(a+1)}}\cdot \bigg(A^{(b-1)}_{a} (Q) -A^{(b-1)}_{a+1} (Q) \bigg).
	 \end{split}
	\end{align} 
\normalsize

We now check that the inductive hypothesis applies to all of the summands in the second summation in (\ref{eqn:second half of 6.9}). For each $q$ with $\rho(q)=d-j+b$, $Q = [\hat{0},q]$ is a $(b-1)$-Sing poset. By (\ref{inductive hyp}), for all $a>\frac{(d-j+b-1)+(b-1)}{2} = \frac{d-j}{2} +b-1$,
	\begin{align}\label{eqn:for Q}
		A^{(b-1)}_a(Q) = (-1)^a \sum_{\rho(t)\leq b-1} e_Q(t,\hat{1})\cdot C(T,\rho(Q)-1,a).  
	\end{align}

In (\ref{eqn:second half of 6.9}), the second summation is a sum over $q\in P$ such that $\rho(q) =d-j+b$ and 
	\[
		a \geq  k-d+\rho(q) -1 = k-j+b-1 > \frac{d+j}{2}-j+b-1=\frac{d-j}{2}+b-1.
	\]
Therefore (\ref{eqn:for Q}) holds for all $A^{(b-1)}_a(Q)$'s in (\ref{eqn:second half of 6.9}), and so
\footnotesize
	\begin{align*}
		A^{(j)}_k &(P)= (-1)^{d-k+1} \sum_{\substack{\rho(t)=r+1\leq j \\ m= \floor{r/2} }}e(t,\hat{1})\cdot C(T,d,d-k)\\
	 	&- (-1)^{k}\sum_{\substack{\rho(q)=d-j+b\\ 0<b\leq j\\a\in [ k-(d-\rho(q))-1 ,\; k-1]}}{d-\rho(Q) \choose {k-(a+1)}}\sum_{\rho(t)\leq b-1} e_Q(t,\hat{1}) \bigg[C(T,\rho(Q)-1,a) +C(T,\rho(Q)-1,a+1) \bigg].
	\end{align*} 
\normalsize
Note that $e_Q(t,\hat{1})=e_P(t,q)$, and by recurence relation (\ref{eq:recurrence}), we obtain:
	\begin{align*}
		A^{(j)}_k(P) &= (-1)^{d-k+1} \sum_{\substack{\rho(t)=r+1\leq j \\ m= \floor{r/2} }}e_P(t,\hat{1})\cdot C(T,d,d-k)\\*
		&- (-1)^{k}\sum_{\substack{\rho(q)=d-j+b\\ 0<b\leq j\\* k-(d-\rho(q))-1 \leq a\leq k-1}}\sum_{\rho(t)\leq b-1} e_P(t,q){d-\rho(Q) \choose {k-(a+1)}}\cdot C(T,\rho(Q),a+1).
	\end{align*} 
Since each $C(T,\rho(Q),a+1)$ is an alternating sum of multiples of the binomial coefficients ${\rho(Q) -\rho(T) \choose (a
+1)-c}$ (for some $c$'s), we can use the Chu-Vandermonde identity to conclude that:
	\[ 
		\sum_{a=k-(d-\rho(q))-1}^{k-1} {d-\rho(Q) \choose k-(a+1)}\cdot {\rho(Q) -\rho(T) \choose (a
	+1)-c}  = {d-\rho(T) \choose k-c}.
	\]
This shows that (for $k>\frac{d+j}{2}$),
\footnotesize
	\begin{align}\label{en after ind hyp}
		\begin{split} 
		A^{(j)}_k & (P) = (-1)^{d-k+1} \sum_{\substack{\rho(t)\leq j }}e_P(t,\hat{1})C(T,d,d-k)- (-1)^{k}\sum_{\substack{\rho(q)=d-j+b\\ 0<b\leq j}}\;\sum_{\rho(t)\leq b-1} e_P(t,q)\cdot C(T,d,k).\\
		&= (-1)^{d-k+1} \sum_{\substack{\rho(t)\leq j }}e_P(t,\hat{1})\cdot C(T,d,d-k) 
		\quad + \quad  (-1)^{k+1}\sum_{\substack{\rho(t) \in [0,\; \rho(q)-(d-j+1)] \\ \rho(q) \in [d-j+1,\; d ]}} e_P(t,q)\cdot C(T,d,k).
		\end{split}
	\end{align} 
\normalsize
The second equality holds because in the second summation, $e_P(t,q)= 0$ for all intervals $[t,q]$ of length $\leq d-j$.

The next step is to apply Corollary \ref{relations of errors for P} to all intervals $[t, \hat{1}]$ with $0\leq \rho(t)\leq j-1$ to replace the summands in (\ref{en after ind hyp}) that involve $e_P(t,q)$ with with sums of multiples of $e_P(t', \hat{1})$ where $0\leq \rho(t')\leq j$.

The cases of even and odd $j$'s are slightly different because of the two cases in Corollary \ref{relations of errors for P}. Here we assume $j$ is odd (and hence $d$ is assumed to be even, see Remark \ref{remark: odd j even d}). The proof for the case of even $j$ is similar; we omit it.

As $j$ is odd and $d$ is even, Corollary \ref{relations of errors for P} implies that
	\begin{align*}
		(-&1)^{k+1} \sum_{\substack{\rho(u)=i \\ d-j+1 \leq \rho(q)\leq d}} e_P(u,q)C(U,d,k) \\
		&=
		\begin{cases}
			(-1)^{k}\sum\limits_{\rho(t)=i}2 e_P(t,\hat{1})\cdot C(T,d,k)+ (-1)^{k}\sum\limits_{\substack{\rho(u)=i\\i+1\leq \rho(t)\leq j\\u<t }} e_P(t,\hat{1})\cdot C(U,d,k) \quad \text{if } i\text{ is even}\\
			(-1)^{k+1}\sum\limits_{\substack{\rho(u)=i\\i+1\leq \rho(t)\leq j\\u<t}} e_P(t,\hat{1})\cdot C(U,d,k)\qquad\qquad \qquad\qquad  \qquad\qquad \quad  \qquad \text{if } i \text{ is odd}.
		\end{cases}
	\end{align*}

After these substitutions, all of the summands in equation (\ref{en after ind hyp}) that involve $e_P(u,q)$ with $q\neq \hat{1}$ are replaced with sums of multiples of $e_P(t,\hat{1})$ for some $t\geq u$. Now we have: 
\footnotesize
	\begin{align*}
		&(-1)^{k+1} \sum_{\substack{0\leq \rho(t)\leq j-1 \\d-j+1\leq \rho(q)\leq d  }} e(t,q)C(T,d,k) \\
		=& (-1)^{k} \sum_{i \text{ even}} \left[\sum\limits_{\rho(t)=i} 2e_P(t,\hat{1})C(T,d,k)+\sum\limits_{\substack{\rho(u)=i\\i+1\leq \rho(t)\leq j\\u<t }} e_P(t,\hat{1})C(U,d,k)\right]  + (-1)^{k+1}\sum_{i \text{ odd}}\sum\limits_{\substack{\rho(u)=i\\i+1\leq \rho(t)\leq j\\u<t}} e_P(t,\hat{1})C(U,d,k)\\
		=& (-1)^{k} \sum_{\rho(t)\leq j} e_P(t,\hat{1})\left[\sum\limits_{u<t}C(U,d,k)\cdot (-1)^{\rho(u)}\right]  + (-1)^{k}\sum_{\rho(t) \text{ is even}} e_P(t,\hat{1})2C(T,d,k).
	\end{align*}
\normalsize
Together with (\ref{en after ind hyp}) and the assumption that $d$ is even, this shows
\footnotesize
	\[
		A_k^{(j)}(P)= \sum_{\rho(t)\leq j} (-1)^{k} e_P(t,\hat{1})\cdot \left[ -C(T,d,d-k) + \sum_{u<t} (-1)^{\rho(u)}C(U,d,k)   \right] + (-1)^{k}\sum_{\substack{\rho(t)\leq j \\\rho(t) \text{ even}}} e_P(t,\hat{1})2C(T,d,k). 
	\]
\normalsize
Comparing this with (\ref{eqn:simplified}), it suffices to show that for each $t\in P$ with $\rho(t)\leq j$ and $T=[\hat{0},t]$,
	\begin{align}\label{coeff. rel}
		(-1)^{k+1}C(T,d,d-k) + (-1)^{k}\sum_{u<t} (-1)^{\rho(u)}C(U,d,k) =
		\begin{cases}
		(-1)^k C(T,d,k) \;&\text{ if } \rho(t) \text{ is odd},\\
		(-1)^{k+1} C(T,d,k)\;& \text{ if } \rho(t) \text{ is even}.
		\end{cases}
	\end{align}

Observe that since d is even,
	\begin{itemize}
		\item $(-1)^{k+1} C(T,d,d-k) = \text{ the coefficient of } x^k \text{ in }  -(x-1)^{d-\rho(t)} \cdot x^{\rho(t)}\cdot \hat{g}(T,\frac{1}{x}), $
		
		\item $(-1)^k\sum\limits_{\substack{u<t }} (-1)^{\rho(u)} C(U,d,k) = \text{ the coefficient of } x^k \text{ in } \sum\limits_{\substack{u<t }} (x-1)^{d-\rho(u)}\cdot \hat{g}(U,x), $
	
		\item $(-1)^k C(T,d,k) = \text{ the coefficient of } x^k \text{ in }   (-1)^{\rho(t)}(x-1)^{d-\rho(t)}\cdot \hat{g}(T,x). $
	\end{itemize}
Therefore, (\ref{coeff. rel}) is equivalent to:
	\begin{align}\label{poly.rel}
		-(x-1)^{d-\rho(t)} \cdot x^{\rho(t)}\cdot \hat{g}(T,\frac{1}{x}) +\sum\limits_{\substack{u<t }} (x-1)^{d-\rho(u)}\cdot \hat{g}(U,x)=-(x-1)^{d-\rho(t)}\cdot \hat{g}(T,x).
	\end{align}
This equality holds since by the definition of $\hat{h}(T,x)$, the left hand-side is 
	\begin{align*}
		& -(x-1)^{d-\rho(t)} \cdot x^{\rho(t)}\cdot \hat{g}(T,\frac{1}{x}) +(x-1)^{d-\rho(t)+1} \cdot \hat{h}(T,x) \\*
		&= (x-1)^{d-\rho(t)} \left[-x^{\rho(t)}\cdot \hat{g}(T,\frac{1}{x})     +(x-1) \cdot \hat{h}(T,x) \right]\\*
		&= - (x-1)^{d-\rho(t)} \cdot \hat{g}(T,x).
	\end{align*}
where the last equality holds since $T$ is Eulerian.

This completes the proof that $A_k^{j} = (-1)^k \sum\limits_{0\leq \rho(t)\leq j}e(t,\hat{1}) C(t)$ when $j$ is odd. When $j$ is even, the proof is very similar; we omit it.
\end{proof}


\subsection{The lower Eulerian case}

In this subsection we assume that $P$ is {\bf lower Eulerian}, i.e., all intervals $[0,t]$ are Eulerian for $t\neq \hat{1}$. This is an important subclass of graded posets. For instance, the face posets of all regular CW complexes are lower Eulerian. If $P$ is $j$-Sing and lower Eulerian, the formula of Theorem \ref{generalized Thm} takes on the following simpler form: 
	\[
		A^{(j)}_k (P) =  \sum_{0\leq \rho(q) =r+1\leq j} \sum\limits_{l=0}^{\floor{r/2}} (-1)^{d-k-l+1} {d-\rho(q) \choose k-\rho(q)-l} \cdot e_P(q,\hat{1}) \cdot \hat{g}_{r-l}([\hat{0},q]) .
	\]

\begin{example}
	Let $P$ be a $j$-Sing poset of rank $d+1$, where $j\leq 2$. If $P$ is also lower Eulerian, then
		\[
			\hat{h}_{d-k}(P) - \hat{h}_k(P) = (-1)^{d-k+1}\sum\limits_{\rho(q)\leq 2} {d-\rho(q) \choose k-\rho(q)}e_P(q,\hat{1}).  
		\]
\end{example}

\begin{rem} 
Unfortunately for larger $j$ the situation becomes more complicated. For instance if $j=3$, then using the fact (easy to check) that
		\[
			\hat{g}(Q) =-\mu(Q)+ [f_1(Q) +\mu(Q) -2]x  \quad\quad \text{ for any poset } Q \text{ with } \rho(Q)=3 
		\]
and the assumption that the poset $P$ is lower Eulerian, one can show that
	\begin{multline}
		A^{(3)}_k (P) =(-1)^{d-k+1}\sum\limits_{\rho(q)\leq 2} {d-\rho(q) \choose k-\rho(q)}e(q,\hat{1}) \\
		+ (-1)^{d-k+1}\sum_{\rho(q) = 3} \underbrace{\bigg[ {d-3 \choose k-3}- {d-3 \choose k-4} \cdot \bigg(f_1([0,q])-3)\bigg)  \bigg] }_{\heartsuit}\cdot e_P(q,\hat{1}).\nonumber
	\end{multline}
Observe that ``$\heartsuit$" equals ${d-3 \choose k-3}$ if and only if $f_1(Q)=3$, which is the case when $Q$ is the face poset of a simplex.
\end{rem}

\begin{definition}
	A pure graded poset $P$ is $k$-{\bf lower simplicial} if  for all $t\in P$ with $\rho(t)\leq k$, the interval $[0,t]$ is a Boolean lattice.
\end{definition}

\begin{corollary}\label{Cor: lower simplicial case}
Let $P$ be a $j$-Sing lower Eulerian poset $P$ of rank $d+1$. If $P$ is also $j$-lower simplicial, then for $k>\floor{\frac{d}{2}}$,
	\[
		\hat{h}_{d-k}(P) - \hat{h}_k(P) = (-1)^{d-k+1}\sum\limits_{\rho(q)\leq j} {d-\rho(q) \choose k-\rho(q)} \cdot e_P(q,\hat{1}).  
	\]
\end{corollary}

\begin{rem}
In the case that $P$ is the face poset of a simplicial complex $\Delta$ and $P$ is $j$-Sing, Corollary \ref{Cor: lower simplicial case} agrees with the formula in Theorem \ref{DSThm}.
\end{rem}


\subsection{Open problems}

The most natural open problem is to find a ``nice" formula for $A^{(j)}_k(P)$ when $j\geq \frac{d}{2}$. Here $\rho(P)=d+1$ and $\frac{d}{2}\leq k \leq \frac{d+j}{2}$.

To state the next problem, let $P$ be a $j$-Sing poset, and let $P^*$ be the dual poset of $P$. By definition, $P^*$ is also a $j$-Sing poset. Recall that
    \begin{itemize}
    \item When $j=-1$, $A^{(-1)}_k (P) = A^{(-1)}_k (P^*)=0.$
    \item When $j=0$, 
      \[ A^{(0)}_k (P) = (-1)^{d-k+1} {d \choose k} e_P(\hat{0},\hat{1}) = A^{(0)}_k (P^*).\]
      When $d$ is odd and $k=\floor{\frac{d}{2}}$, this means
      \[ \hat{g}_{\floor{\frac{d}{2}}} (P) = \hat{g}_{\floor{\frac{d}{2}}} (P^*). \]
    \item When $j=1$ (and assuming $d$ is even), after cancellations, we obtain
    \[ A^{(1)}_k (P) -A^{(1)}_k (P^*) =(-1)^{d-k}{d-1 \choose k} \bigg[\sum_{\rho(q)=1}e_P(q,\hat{1}) - \sum_{\rho(Q)=d}e_P(\hat{0},q)\bigg]. \] 
    \end{itemize}
This leads to the following question:
\begin{question}
	How do the numbers $A^{(j)}_k (P)$ compare with $A^{(j)}_k (P^*)$?
\end{question}

\section*{Acknowledgments}
The authors would like to thank Ed Swartz for bringing up the problem on Dehn--Sommerville relations of toric $h$-vectors and also for the suggestion on using the short $h$-vectors for the second proof of Theorem \ref{DSThm}. We also thank Isabella Novik for having posed the problem to us and having taken many hours to discuss and to help revise the first few drafts of this paper, and Bennet Goeckner for numerous comments on the draft. The second author is also grateful to Margaret Bayer for valuable questions at the AMS meeting in Florida. We are also thankful to the referee for providing helpful suggestions on how to strengthen our results, which led to a significant improvement of Section \ref{sect:flag DS rel}.

\bibliographystyle{alpha}
\bibliography{dehn_sommerville_biblio}
\end{document}